\documentclass[a4paper,11pt]{amsart}

\usepackage{hyperref}
\usepackage{amsfonts, amssymb}
\usepackage{amsmath}
\usepackage{amsthm}
\usepackage[foot]{amsaddr}
\usepackage{graphicx}
\usepackage{color} 
\usepackage{aliascnt} 
\usepackage{float}
\usepackage{tikz}
\usepackage{mathtools}
\usepackage{enumitem}
\usepackage{stmaryrd}
\usepackage[makeroom]{cancel}
\usepackage{natbib}

\usepackage[a4paper]{geometry} 
\newtheorem{theorem}{Theorem}[section]
\newaliascnt{lemma}{theorem}  
\newtheorem{lemma}[lemma]{Lemma}  
\aliascntresetthe{lemma}  
\newaliascnt{corollary}{theorem}  
 
\aliascntresetthe{corollary}
\newaliascnt{proposition}{theorem}  
\newtheorem{proposition}[proposition]{Proposition} 
\aliascntresetthe{proposition}
\newaliascnt{definition}{theorem}  
\newtheorem{definition}[definition]{Definition} 
\aliascntresetthe{definition}
\newaliascnt{remark}{theorem}  
\newtheorem{remark}[remark]{Remark}  
\aliascntresetthe{remark}

\newcommand{\abs}[1]{\vert #1 \vert}
\newcommand{\norm}[1]{\Vert #1 \Vert}
\newcommand{\R}{\mathbb{R}} 

\newcommand{\Z}{\mathbb{Z}}
\newcommand{\N}{\mathbb{N}}
\def\x {\,\mathrm{d}x}
\def\y {\,\mathrm{d}y}

\def\t {\,\mathrm{d}t}
\def\s {\,\mathrm{d}s}

\def\Sy {\,\mathrm{d}\sigma(y)}

\newcommand{\eps}{{\varepsilon}}
\newcommand{\OmEps}{\Omega_\varepsilon}
\newcommand{\TEps}{\mathcal{T}_\varepsilon}

\newcommand{\weak}{\rightharpoonup}

\newcommand{\GamD}{{\Gamma_\mathrm{D}}}
\newcommand{\GamN}{{\Gamma_\mathrm{N}}}
\newcommand{\Hgam}{H^1_{\Gamma_\mathrm{D}}}

\newcommand{\spaceH}{\mathcal{H}}

\newcommand{\vareps}{\varepsilon}
\newcommand{\ueps}{u_{\varepsilon}}
\newcommand{\weps}{w_{\varepsilon}}

\newcommand{\foe}{\frac{1}{\varepsilon}}

\newcommand{\lay}{\mathrm{M}}

\title{Effective elastic wave transmission through a periodically voided interface}

\author[]{M.~Gahn$^{*}$, T.~Lochner$^{*}$, M.~A.~Peter$^{\dagger}$ }

\address{${}^{*}$ Institute of Mathematics, University of Augsburg, Universit\"atsstr.~12a, 86159 Augsburg, Germany}
\address{${}^{\dagger}$ Institute of Mathematics \& Centre for Advanced Analytics and Predictive Sciences (CAAPS), University of Augsburg, Universit\"atsstr.~12a, 86159 Augsburg, Germany}

\email{markus.gahn@uni-a.de (ORCID:~0000-0002-2216-145X)}
\email{tanja.lochner@uni-a.de}
\email{(ORCID:~0000-0003-1962-5724), malte.peter@uni-a.de (ORCID:~0000-0001-6107-9806)}

\date{\today}

\begin{document}

\begin{abstract}
   Effective interface conditions for a periodically voided thin layer separating two homogeneous bulk regions are derived for the elastic wave equation by taking the simultaneous limit of vanishing layer periodicity and layer thickness. The limit problems are obtained using the unfolding method for thin perforated domains. We consider three different scalings of the material parameters in the layer that characterise its stiffness, each leading to a distinct type of interface condition and requiring the solution of scaling‑dependent cell problems. Depending on the scaling, the resulting effective model yields either a membrane equation or a Kirchhoff–Love plate equation. In the critical regime of reduced stiffness, the interface equation additionally depends on the microscopic variable. By selecting appropriate cell problems, this equation can be reformulated as an effective interface condition between the bulk domains.
\end{abstract}

\maketitle


\noindent
{\bf Mathematics Subject Classification (MSC):} 35B27, 35L53, 35Q74, 74Q10

\section{Introduction}



The interaction of waves with microstructured layers has become an increasingly active research area. Applications include the use of layers to guide waves or to reduce or enhance transmission or reflection. Many of these developments have been framed in the context of metamaterials, which are materials whose small-scale internal architecture is deliberately designed to produce novel and remarkable behaviours, e.g.~invisibility cloaks, superlenses and wave-guiding devices \cite{davies2025roadmap,craster2023mechanical,Chaplain2025}.
 The use of single periodic layers of inclusions or voids embedded into a matrix of other material is of particular relevance for absorption of wave energy, leading to decreased reflection, as e.g.~in submarine coatings (``Alberich tiles''), and such structures have gained increased interest again, see e.g.~\cite{Skvortsov2021,PhysRevB.102.214308} and references therein. In the long-wavelength regime (compared to the size of the inclusions), homogenisation techniques are the method of choice to derive effective boundary/transmission conditions by considering a single layer of periodically distributed inclusions or channels and passing to the limit as the periodicity length and the thickness of the layer vanish simultaneously, see e.g.~\cite{DELOURME201228}. We also note that higher order approaches taking into account finite size effects have been developed for such acoustic problems \cite{Marigo_Maurel2017}.

Before going into the details of the challenges of the rigorous mathematical treatment of such problems, we introduce the precise system we consider here, which is given by the elastic wave equations 
\begin{align*}
\partial_t(\varrho^\pm\partial_t u_\eps^\pm)-\nabla\cdot(A^\pm e(u_\eps^\pm))&=\varrho^\pm f^\pm &&\text{ in }(0,T)\times\OmEps^\pm,\\
         \partial_t(\frac{1}{\eps}\varrho_\eps^\lay\partial_t u_\eps^\lay)-\nabla\cdot(\eps^{\gamma} A_\eps^\lay e(u_\eps^\lay))&=\frac{1}{\eps}\varrho_\eps^\lay f_\eps^\lay &&\text{ in }(0,T)\times\OmEps^\lay,
\end{align*}
describing the evolution of the bulk displacements $\ueps^{\pm}$ in the bulk domains $\OmEps^{\pm}$ and the displacement $\ueps^\lay$ in the thin perforated layer $\OmEps^\lay$ separating $\OmEps^{+}$ from $\OmEps^{-}$. The parameter $\gamma $ is related to the stiffness of the material and we consider the cases $\gamma \in \{1,-1,-3\}$. The interface between the bulk domain $\OmEps^{\pm}$ and the thin layer is denoted by $S_{\eps}^{\pm}$ and continuity of the traces and the normal stresses are assumed on this interface:
\begin{align*}
	u_\eps^\pm &= u_\eps^\lay &&\text{ on }(0,T)\times S_\eps^\pm,\\
        (A^{\pm} e(u_\eps^{\pm}))n &=\eps^{\gamma}(A^\lay_\eps e(u^\lay_\eps))n &&\text{ on }(0,T)\times S^{\pm}_\eps.
\end{align*}
The system is closed by suitable initial and boundary conditions, see Section \ref{sec:statement_problem} for more details. 

For $\eps \to 0$ the thin layer $\OmEps^\lay$ reduces to a lower dimensional interface $\omega$ separating the bulk domains $\Omega^+$ and $\Omega^-$. The main objective of this paper is the rigorous derivation of a macroscopic model for $\eps \to 0$ together with suitable interface  conditions across $\omega$. 
While it is natural to expect elastic wave equations for the macroscopic displacement $u_0^{\pm}$ in the bulk regions, the key challenge lies in identifying the effective interface conditions, which, of course, depend on the specific choice of $\gamma$.

For the derivation of the macroscopic model, we use the unfolding method, see the monograph \cite{Cior:2018} for an overview. While the sequence of displacements $\ueps^\lay$ in the thin layer depends on the $\eps$-dependent domain $\OmEps^\lay$, the associated unfolded sequence $\TEps(\ueps)$ is defined on the fixed domain $\omega \times Y_0$, with the reference element $Y_0$ representing the  geometry of the porous part of the micro-domain. 
Compactness results for the unfolded sequence are based on uniform a priori estimates with respect to the scaling parameter $\eps$. For this, we have to use different types of Korn inequalities. For $\gamma =1,-1$, it turns out that we can control the displacement in the thin layer by the bulk displacement and its symmetric gradient. For the case $\gamma = -3$, this approach does not work and we use a Korn inequality which allows to control the displacement in the layer by its symmetric gradient only at the price of having to consider a clamped plate, i.e.~we impose a vanishing Dirichlet boundary condition. Based on these a priori estimates, we obtain by well-known (two-scale) compactness results the weak convergence of the unfolded sequence in the thin layer, where we also require the assumption of a clamped plate for the case $\gamma = -1$. Here, the important difference between the cases $\gamma =1$ and $\gamma = -1,-3$ emerges. While in the latter case, the limit displacement (of zeroth order) is independent of the microscopic variable $y \in Y_0$, in the case $\gamma = 1$ it depends on the macroscopic variable $x' \in \omega$ \textit{and} the microscopic variable $y$. For the case $\gamma = -3$, we obtain a Kirchhoff--Love (plate) displacement as the limit of the unfolded sequence.

The next step is to pass to the limit $\eps\to 0$ in the weak microscopic equation and to identify the complete limit problem. For this, we have to choose suitable test functions adapted to the structure of the limit displacements. For $\gamma = -1$, we obtain continuity of the bulk displacements $u^+$ and $u^-$ across the interface $\omega$ equal to the limit of the displacement in the layer $u^\lay$. The jump of the normal stress is determined by a membrane equation for $u^\lay$ containing effective coefficients carrying information about the microscopic geometry of the perforated layer. For the case $\gamma  = -3$, we again obtain continuity of the bulk displacements, given by the zeroth-order limit of the unfolded sequence in the layer. Now, the jump of the normal stress is given by a a Kirchhoff--Love plate equation with effective coefficients. The situations changes completely in the case $\gamma = 1$, when the limit of the layer displacement depends additionally on the microscopic variable, that is we obtain a two-scale homogenised model for $u^{\pm}$ and $u^\lay$ in the limit. Here, in every macroscopic point $x'\in \omega$ we have to solve a cell problem for $u^\lay$, depending on the time derivative of the trace of the bulk displacements, i.e.~$\partial_t u^{\pm}|_{\omega}$. Now, the idea is to find a representation for $u^\lay$ via suitable cell problems with coefficients given by the force terms in the layer and the traces of the bulk displacements. While this procedure is well-known for periodic domains, the identification of this representation is considerably more complicated in the case of the thin layer, in particular regarding the hyperbolic character of the cell problem for $u^\lay$ coupled to the bulk solutions on the boundary. Finally, using this representation, we can use the two-scale homogenised problem together with suitable test functions to obtain an effective interface condition. As a result, we find a wave equation for each of the bulk displacements $u^+$ and $u^-$, where the normal stress (not the jump) is given by an effective boundary condition dependent on $u^+$ and $u^-$ (and their time derivatives) characterised by effective coefficients containing a memory effect. From a physical point of view, the thin layer behaves like a thick interface in the limit in this case.

From a mathematical point of view, our results are generalisations of \cite{Neuss:2007} and \cite{GahnEffectiveTransmissionContinuous} to elastic materials. More precisely, in \cite{Neuss:2007} a system of (nonlinear) reaction--diffusion equations was considered for the case $\gamma = 1$, where a critical point was to establish strong two-scale convergence in the thin layer. The structure of the obtained two-scale homogenised system is similar to our limit model (for $\gamma = 1$), where in our contribution we also derive the effective interface conditions across $\omega$. We emphasise that this is only possible here since our problem is linear. In \cite{GahnEffectiveTransmissionContinuous}, the case $\gamma = -1$ was considered for a reaction--diffusion equation (together with some intermediate results for $\gamma \in (-1,1)$). In \cite{GahnNeussRaduKnabner2018a}, these results were extended to a nonlinear coupling condition for the fluxes across the interfaces between the thin layer and the bulk regions. In~\cite{ClopeauFerrinGilbertc2001,GilbertMikelic2000}, homogenisation was carried out for a (linearised) fluid--structure interaction problem under varying ratios between the viscosity and the characteristic scale of the elasticity coefficients. This led to different scalings of the governing equations with respect to~$\varepsilon$. While their analysis was conducted in a fully periodic setting, the decomposition of the limit functions in the case $\gamma = 1$ closely resembles our situation and similarly gives rise to a memory effect. However, our model additionally requires accounting for the coupling to the bulk regions and the fact that no periodicity is assumed in the $x_3$-direction (the direction normal to the layer)---a structural distinction that has not yet been made explicit but plays a crucial role in our setting.
In \cite{gahn2025effective}, fluid flow through a thin (rigid) perforated layer separating two bulk regions was considered (similar to our geometry) with linearised fluid--structure interaction (together with a reaction--diffusion--advection problem), and the simultaneous homogenisation and dimension reduction was performed, leading to an effective interface condition for the fluid flow. This model was generalised to fluid flow through a thin elastic perforated layer in \cite{gahn2025b_effective} for $\gamma = -1$ and $\gamma = -3$. In our paper, we use similar ideas 
for the identification of the effective interface conditions. We also mention the work \cite{OrlikPanasenkoStavre2021}, where the fluid-- structure interaction between two fluid-filled bulk domains and a thin (non-perforated) elastic layer is investigated for $\eps \to 0$. The case of an elastic perforated thin layer, not coupled to bulk regions, with a scaling similar to our case $\gamma = -3$ was treated in \cite{griso2020homogenization}. In a more general context, we also mention that homogenisation problems in which a thin perforated interface is homogenised are sometimes referred to as sieve problems, see e.g.~\cite{Cior:2018}, \cite{Schweizer:2020}. Moreover, related inverse homogenisation problems, in which the geometry of the microscopic inclusions/voids is to be determined from macroscopic displacement measurements have recently been considered in \cite{Loch:2022,Loch:2024}.

Before giving the complete problem statement including all assumptions in \S \ref{sec:statement_problem}, we close this Section with a few notes on notation. The main results and the obtained macroscopic model are summarised in \S \ref{sec:mainresults}, distiguishing the cases $\gamma = 1$, $\gamma = -1$ and $\gamma = -3$ in \S\ref{sec:mainresults_gam1}, \S\ref{sec:mainresults_gam-1} and \S\ref{sec:mainresults_gam-3}, respectively. Existence and uniform boundedness of the microscopic solutions is proven in a unified way in \S\ref{sec:existence_boundedness} before the resulting compactness results for the sequences of microscopic solution are discussed in \S\ref{sec:compactness}. Finally, the limit problems are identified in \S\ref{sec:identification}.

\subsection{Notation}
For an open set $\Omega \subset \R^n$  and $f : \Omega \to \R^m$ for $n,m \in \N$, we  denote the derivative  of $f$ by $\partial f$ and identify it with its Jacobian matrix given by $(\partial f(x))_{ij} \coloneqq  \partial_{x_j} f_i(x) = \partial_j f_i(x)$ for $i=1,\ldots,n$, and $j=1,\ldots,m$. Further, we define  the gradient of $f\colon \Omega \rightarrow \R^m$ as the transpose of the Jacobian matrix, $\nabla f \coloneqq \partial f^\top$. For $n=m$ we define the symmetric gradient by
\begin{align*}
    e(f):=\frac12 \left(\nabla f + \nabla f^T\right).
\end{align*}
For two matrices $B,C \in \R^{n\times n}$ we define  the Frobenius-product $B:C\coloneqq \mathrm{tr}(B^{\top} C) =  \sum_{i,j=1}^n B_{ij}C_{ij}$.
Let $n,d\in \N$, then for $\Omega\subset \R^n$ a bounded Lipschitz domain, we denote by $L^p(\Omega)^d, \, W^{1,p}(\Omega)^d $ the standard Lebesgue and Sobolev spaces with $p \in [1,\infty]$. In particular, for $p=2$, we write $H^1(\Omega)^d\coloneqq W^{1,2}(\Omega)^d$.  For the norms, we omit the upper index $d$, for example we write $\|\cdot\|_{L^p(\Omega)}$ instead of $\|\cdot \|_{L^p(\Omega)^d}$. With $\Lambda$ a subset of $\partial \Omega$, we let $H^1_\Lambda(\Omega)$ denote the $H^1(\Omega)$ functions vanishing on $\Lambda$ (in the sense of traces). 
For a separable Banach space $X$ and $p \in [1,\infty]$, we denote the usual Bochner spaces by $L^p(\Omega,X)$, 
or $L^p((0, T), X)$ when time is involved. 
For the dual space of $X$, we use the notation $X'$.

\section{Statement of the problem}\label{sec:statement_problem}

We consider elastic wave propagation through a domain made of two (possibly different) materials, which are separated by a thin interface of $\eps$-periodically voided elastic medium. Let $\omega\subset\R^2$ be a connected domain with Lipschitz boundary, $\Omega=(-L,L)\times\omega$, $L>0$ fixed, $\eps>0$ and $S=(0,T)$ with final time $T>0$. We also identify $\omega $ with its natural embedding $\{0\} \times \omega $ in $\R^3$. We divide the domain $\Omega$ into 
$$\Omega^{-}=(-L,0)\times\omega,\quad \Omega^{+}=(0,L)\times\omega,\quad \{0\} \times \omega$$
(see \autoref{subfig:domain_omega} for an illustration) and into
$$\OmEps^{-}=\left(-L,-\frac{\eps}{2}\right)\times\omega,\quad \OmEps^{+}=\left(\frac{\eps}{2},L\right)\times\omega,\quad \OmEps^{0}=\left(-\frac{\eps}{2},\frac{\eps}{2}\right)\times\omega,\quad S_{\eps}^\pm=\Bigl\{\pm\frac{\eps}{2}\Bigl\}\times\omega$$
(see \autoref{subfig:domain_omega_eps}).
\begin{figure}[t]
\begin{center}
\begin{minipage}[t]{0.3\textwidth}
\centering
    \begin{tikzpicture}[scale=2.6]
        \draw (-1/2,0) rectangle (1/2,1);
        \draw[magenta] (0,0) -- (0,1);
        \draw (-1/4,0.5) node {$\Omega^{-}$};
        \draw[magenta] (-0.07,0.8) node {$\omega$};
        \draw (1/4,0.5) node {$\Omega^{+}$};
        \draw (-1/2,0) node[below] {$-L$};
        \draw (0,0) node[below] {$0$};
        \draw (1/2,0) node[below] {$L$};    
    \end{tikzpicture}
    \caption{Illustration of the domain $\Omega$}
    \label{subfig:domain_omega}
\end{minipage}
$\quad$
\begin{minipage}[t]{0.3\textwidth}
\centering
\begin{tikzpicture}[scale=2.6]
    \draw (-3/4,0) rectangle (3/4,1);
    \draw[magenta] (-0.14,0) -- (-0.14,1);
    \draw[magenta] (0.14,0) -- (0.14,1);
    \draw (-0.55,0.5) node {$\OmEps^{-}$};
    \draw[magenta] (-0.2,0.8) node {$S_{\eps}^{-}$};
    \draw (0,0.5) node {$\OmEps^{0}$};
    \draw[magenta] (0.2,0.8) node {$S_{\eps}^{+}$};
    \draw (0.55,0.5) node {$\OmEps^{+}$};
    \draw (-3/4,0) node[below] {$-L$};
    \draw (-0.14,0) node[below] {$-\frac{\eps}{2}$};
    \draw (0.14,0) node[below] {$\frac{\eps}{2}$};
    \draw (3/4,0) node[below] {$L$};    
\end{tikzpicture}

    \caption{Subdomains of $\Omega$ depending on $\eps$}
    \label{subfig:domain_omega_eps}
\end{minipage}
$\quad$
\begin{minipage}[t]{0.3\textwidth}
\centering
    \begin{tikzpicture}[scale=2.5]
	   \draw (-1/2,0) rectangle (1/2,1);	
	   \draw (0,0.5) circle [x radius=0.35cm, y radius=0.17cm];
        \draw[magenta] (-1/2,0) -- (-1/2,1);
        \draw[magenta] (1/2,0) -- (1/2,1);
	   \draw (0,0.5) node {$Y_1$};
	   \draw (0.2,0.14) node {$Y_0$};
        \draw (-1/2,0) node[below] {$-\frac{1}{2}$};
        \draw (1/2,0) node[below] {$\frac{1}{2}$};
        \draw[magenta] (-0.59,0.8) node {$S_{Y}^{-}$};
        \draw[magenta] (0.59,0.8) node {$S_{Y}^{+}$};
    \end{tikzpicture}
    \caption{Reference cell $Y$}
    \label{subfig:reference_cell}
\end{minipage}
\end{center}
\end{figure}
To model the periodic in-plane structure of $\OmEps^0$, we define the reference cell $Y=(-\frac{1}{2},\frac{1}{2})\times(0,1)^2$, the holes $Y_1\subset Y$ open with $\overline{Y_1}\subset Y$, $Y_0=Y\setminus Y_1$, $Y'=(0,1)^2$ and $S_{Y}^\pm=\Bigl\{\pm\frac{1}{2}\Bigl\}\times Y'$ (see \autoref{subfig:reference_cell}).
Denote by
\begin{itemize}
    \item $\Lambda_{\eps}=\{\xi\in\Z^2:\eps(\xi+Y')\subset\omega\}$, $\quad \Lambda^{0}_{\eps}=\{0\}\times\Lambda_{\eps}$
\end{itemize}
and assume $\overline{\omega} = \bigcup_{k\in\Lambda_\eps}{\eps (\overline{Y}+k)} $ for all $\eps$ small enough.
With this at hand, we can define the layer $\OmEps^\lay$ and the domain $\OmEps$ as described in the motivation as
\begin{itemize}
   \item $\OmEps^\lay=\OmEps^{0}\setminus\overline{\{ x\in \Omega^{0}_{\eps}:\{\frac{x}{\eps}\}_Y\in Y_1\}}$, 
   \item $\OmEps=\Omega\setminus\overline{\{ x\in\Omega^{0}_{\eps}:\{\frac{x}{\eps}\}_Y\in Y_1\}}$.
\end{itemize}
Here and in what follows, $\{{x}\}_Y$ is the part inside the reference cell $Y$, i.e.~for $x\in\Omega^{0}_{\eps}$, we write $\left[{x}\right]_Y$ for the unique linear combination of the unit vectors $e_j\in\mathbb{R}^3$ with integer coordinates $\xi_j\in\mathbb{Z}$, $j\in\{2,3\}$, i.e.~$\left[{x}\right]_Y=\sum_{j=2}^3{\xi_j e_j}$, such that $\{{x}\}_Y = {x}-\left[{x}\right]_Y\in Y$.
The restrictions of functions defined on $\Omega$ to the subdomains $\OmEps^{+}$, $\OmEps^\lay$, $\OmEps^{-}$ are denoted by the superscripts $+$, $-$ and $M$, respectively. 

\begin{remark}
For the sake of simplicity, we assumed that $\omega$ is a rectangular domain with integer side-length and $\varepsilon$ is chosen in such a way that the perforated layer $\Omega_{\vareps}^{\lay}$ can be completely decomposed in shifted and translated micro-cells $\vareps (Y_0 + k)$ with $k\in \{0\} \times\mathbb{Z}^2$. In particular, no perforations intersect the lateral boundary of $\Omega_{\varepsilon}^0$. However, it is straightforward to generalise our results to curved domains $\omega$, with the additional assumption that there is some kind of "safety zone" around the boundary $\partial \omega$ where no perforations intersect the lateral boundary. 
\end{remark}

Before stating the problem to be considered in what follows, we introduce the periodic unfolding operator for future reference. Such an operator was first defined (not using the name unfolding operator common nowadays) for domains in  \cite{VogtHomogenization} and later analysed in more detail for example in \cite{ArbogastDouglasHornung,Cioranescu_Unfolding2}. An unfolding operator for thin domains (also called partial unfolding operator) was first introduced in \cite{Neuss:2007}, see also \cite{CioranescuDamlamianGrisoOnofrei2008,Griso:2015}. We also refer to \cite{Cior:2018} for a detailed overview about this topic.
\begin{definition}\label{def:unfolding_operator}
    For $\phi$ Lebesgue-measurable on $\Omega^\lay_{\eps}$, we define
    \begin{align*}
     \TEps(\phi)(x',y)=
		\phi(\eps\left[\frac{(0,x')}{\eps}\right]_Y+\eps y) \quad &\text{ for a.e. }(x',y)\in \omega \times Y_0
    \end{align*}
\end{definition}

Now, let $\eps\ll L$, $\GamD\subset\partial\Omega$ be a closed subset with positive (two-dimensional) measure and $\GamN=\partial\Omega\setminus\GamD$.
We define for every $\eps$ the (vector-valued) three-dimensional wave-propagation problem 
\begin{subequations}
\begin{equation}\label{eq:strong_formulation}
 	\begin{aligned}
 		\partial_t(\varrho^\pm\partial_t u_\eps^\pm)-\nabla\cdot(A^\pm e(u_\eps^\pm))&=\varrho^\pm f^\pm &&\text{ in }S\times\OmEps^\pm,\\
         \partial_t(\frac{1}{\eps}\varrho_\eps^\lay\partial_t u_\eps^\lay)-\nabla\cdot(\eps^{\gamma} A_\eps^\lay e(u_\eps^\lay))&=\frac{1}{\eps}\varrho_\eps^\lay f_\eps^\lay &&\text{ in }S\times\OmEps^\lay,
 	\end{aligned}
\end{equation}
where $e(u)=\frac{1}{2}(\nabla u+(\nabla u)^T)$ is the symmetrised gradient, $\gamma \in \{1,-1,-3\}$ is related to the stiffness of the thin layer,  $A^\eps$ is the elasticity tensor of fourth order,
\[ A^\eps(x)=\begin{cases}
        A^\pm(x), &x\in\OmEps^\pm,\\
        \eps^{\gamma} A^\lay\left(x',\frac{x}{\eps}\right), &x\in\OmEps^\lay,
    \end{cases} \] 
and we write $A_{\eps}^\lay(x):= A^\lay\left(x',\frac{x}{\eps}\right)$. Moroever,
$f^\eps$ is the body load (force per mass),
\[ f^\eps(t,x)=\begin{cases}
        f^\pm(t,x), &x\in S\times\OmEps^\pm,\\
        f^\lay_\eps(t,x), &x\in S\times\OmEps^\lay,
    \end{cases} \]
and $\varrho^\eps$ is the (mass density) function,
\[ \varrho^\eps(x)=\begin{cases}
        \varrho^\pm(x), &x\in\OmEps^\pm,\\
        \frac{1}{\eps} \varrho^\lay\left(x',\frac{x}{\eps}\right), &x\in\OmEps^\lay,
    \end{cases} \]
where, again, we write $\varrho_{\eps}^\lay(x):= \varrho^\lay\left(x',\frac{x}{\eps}\right)$. Further, we prescribe the boundary conditions
\begin{equation}\label{eq:boundary_cond}
	\begin{aligned}
		\ueps &= 0 &&\text{ on }S\times\GamD,\\
		(A^\eps e(\ueps))\nu &= g &&\text{ on }S\times\GamN,\\
		(A^\eps e(\ueps))\nu &= 0 &&\text{ on }S\times\partial\OmEps\setminus\partial\Omega,
	\end{aligned}
\end{equation}
at the exterior and void boundaries, respectively,
where $\nu$ is the outward-pointing normal to $\GamN$ resp.~$\partial\OmEps\setminus(\GamN\cup\GamD)$ and $g$ is a boundary force. The bulk domains and the layer are coupled by natural transmission conditions, i.e.~the continuity of the displacements and of the normal stress vectors,
\begin{equation}\label{eq:transmission_cond}
	\begin{aligned}
		u_\eps^\pm &= u_\eps^\lay &&\text{ on }S\times S_\eps^\pm,\\
        (A^{\pm} e(u_\eps^{\pm}))n &=\eps^{\gamma}(A^\lay_\eps e(u^\lay_\eps))n &&\text{ on }S\times S^{\pm}_\eps,
	\end{aligned}
\end{equation}
where $n$ is the unit normal with direction from $\OmEps^\lay$ to $\OmEps^\pm$, 
Finally, the initial conditions
\begin{equation}\label{eq:initial_cond}
	\begin{array}{rl}
	\ueps(0,x) &=u^\eps_0(x)\coloneq\begin{cases}
        u^\pm_0(x), &x\in\OmEps^\pm,\\
        u^\lay_{0,\eps}(x), &x\in\OmEps^\lay,
        \end{cases}\\
	\partial_t \ueps(0,x) &=u^\eps_1(x)\coloneq\begin{cases}
        u^\pm_1(x), &x\in\OmEps^\pm,\\
        u^\lay_{1,\eps}(x), &x\in\OmEps^\lay,\\
    \end{cases}
	\end{array}
\end{equation}
satisfying the compatibility conditions
\begin{equation}\label{eq:compatibility_cond}
\begin{aligned}
    u^{\pm}_0 & =u^\lay_{0,\eps} &&\text{on }S^{\pm}_\eps,\\
    (A^{\pm} e(u^{\pm}_0))n &=\eps^{\gamma}(A^\lay_\eps e(u^\lay_{0,\eps}))n &&\text{on }S^{\pm}_\eps,\\
    u^{\pm}_0 &= 0 &&\text{on }\GamD\cap\partial\OmEps^\pm,\\
    u^\lay_{0,\eps} & = 0 &&\text{on }\GamD\cap\partial\OmEps^\lay.
\end{aligned}
\end{equation}
\end{subequations}
are assumed.

Before we give the definition of a weak solution of this microscopic problem, we introduce some Hilbert spaces with suitable norms and inner products adapted to the structure of the micromodel.
We define the spaces
\[ \Hgam(\OmEps)\coloneqq\left\lbrace u\in\left[H^1(\OmEps)\right]^3 \, | \, u=0\text{ on }\GamD\right\rbrace\quad\text{and}\quad L^2_{\varrho^\eps}(\OmEps)\coloneqq\left[L^2(\OmEps)\right]^3 \]
equipped with norms
\[ \norm{u}_{\Hgam(\OmEps)} =\norm{e(u)}_{\left[L^2(\OmEps)\right]^{3\times 3}}\quad\text{and}\quad\norm{u}_{L^2_{\varrho^\eps}(\OmEps)}=\sqrt{\langle u,u\rangle_{L^2_{\varrho^\eps}(\OmEps)}},\]
where $\langle \cdot,\cdot\rangle_{\varrho^\eps}$ is the weighted inner product 
\[ \langle u,v\rangle_{L^2_{\varrho^\eps}(\OmEps)} \coloneqq \int_{\OmEps} \varrho^\eps(x)u(x)v(x)\x \]
on the space $L^2_{\varrho^\eps}(\OmEps)\times L^2_{\varrho^\eps}(\OmEps)$. In fact, $\norm{\cdot}_{\Hgam(\OmEps)}$ is a norm on $\Hgam(\OmEps)$ because of Korn's inequality for functions with vanishing trace on part of the boundary. Then, there holds
\[ \Hgam(\OmEps)\subset L^2_{\varrho^\eps}(\OmEps)=\left(L^2_{\varrho^\eps}(\OmEps)\right)^*\subset (\Hgam(\OmEps))^*, \]
and we note that $\Hgam(\OmEps)$ is a separable Hilbert space. When we write $L^2(\mathcal{O})$ for some open set $\mathcal{O}$, we equip this Hilbert space with the standard norm $\norm{u}^2_{L^2(\mathcal{O})}=\int_\mathcal{O}\abs{u}^2\x$. Now, we give the definition of a weak solution of the microscopic problem:

\begin{definition}\label{def:weak_micro_solution}
We call $\ueps :(0,T)\times \OmEps \rightarrow \R^3$ a weak solution of the micromodel $\eqref{eq:strong_formulation}$ with boundary, transmission, initial and compatibility conditions \eqref{eq:boundary_cond}--\eqref{eq:compatibility_cond}, if $\ueps\in L^2(S;\Hgam(\OmEps))$ with  $\partial_t \ueps\in L^\infty(S;L^2(\OmEps))$ 
such that 
\begin{equation}\label{eq:weak_formulation}
\begin{split}
	-&\sum_{\pm}\left\{\int_0^T\int_{\OmEps^\pm} \varrho^\pm\partial_t u^\pm_\eps\cdot\partial_t v\,\x\t -\int_0^T\int_{\OmEps^\pm} A^\pm e(u^\pm_\eps):e(v)\,\x\t   \right\} \\ 
    & -\int_0^T\int_{\OmEps^\lay} \frac{1}{\eps}\varrho^\lay_\eps\partial_t u^\lay_\eps\cdot\partial_t v\,\x\t +\int_0^T\int_{\OmEps^\lay} \eps^\gamma A^\lay_\eps e(u^\lay_\eps):e(v)\,\x\t\\
	=& \sum_{\pm} \left\{\int_0^T\int_{\OmEps^\pm} \varrho^\pm f^\pm\cdot v\,\x\t + \int_{\OmEps^\pm}\varrho^\pm u^\pm_1\cdot v(0)\,\x +\int_0^T\int_\GamN g\cdot v\,\mathrm{d}\sigma(x)\t \right\} \\
    &+\int_0^T\int_{\OmEps^\lay} \frac{1}{\eps}\varrho^\lay_\eps f^\lay_\eps\cdot v\,\x\t + \int_{\OmEps^\lay}\frac{1}{\eps}\varrho^\lay_\eps u^\lay_{1,\eps}\cdot v(0)\,\x
\end{split}
\end{equation}
for all $v\in L^2(S;\Hgam(\OmEps))$ with $\partial_t v\in L^2(S;L^2(\OmEps))$ and $v(T)=0$. Further, the initial conditions $\ueps(0) = u_0^{\eps}$ and $\partial_t \ueps (0) = u_1^{\eps}$ are to be satisfied.
\end{definition}

 The short notation of \eqref{eq:weak_formulation} is
 \begin{equation}\label{eq:weak_formulation_short}
 	\begin{split}
 		-\int_0^T\int_{\OmEps}& \varrho^\eps\partial_t \ueps\cdot\partial_t v\,\x\t+\int_0^T\int_{\OmEps} A^\eps e(\ueps):e(v)\,\x\t\\
 		&= \int_0^T\int_{\OmEps} \varrho^\eps f^\eps\cdot v\,\x\t +\int_0^T\int_\GamN g\cdot v\,\mathrm{d}\sigma(x)\t + \int_{\OmEps}\varrho^\eps u^\eps_1\cdot v(0)\,\x.
 	\end{split}
 \end{equation}

Finally, we list the assumptions on the data ensuring well-posedness of the problem and allowing the homogenisation analysis.

\noindent\textbf{Assumptions on the data:}
\begin{enumerate}
[label = (A\arabic*)]
\item 
Let $\alpha,\beta\in\mathbb{R}$ with $0<\alpha<\beta$. Denoting by $M(\alpha,\beta,\mathcal{O})$ for an open set $\mathcal{O}\subset \mathbb{R}^3$ the set of all tensors $B = (b_{ijkh})_{1\leq i,j,k,h\leq 3}$ such that 
\begin{enumerate}[label = (\roman*)]
	\item{$b_{ijkh}\in L^{\infty}\left(\mathcal{O}\right)\text{ for all } i,j,k,h\in\{1,2,3\}$,}
	\item{$b_{ijkh}=b_{jikh}=b_{khij}\text{ for all } i,j,k,h\in\{1,2,3\}$,}
	\item{$\alpha\vert m\vert^2\leq Bm:m$ for all symmetric matrices $m$,}
	\item{$\vert B(x)m\vert\leq\beta\vert m\vert$ for all matrices $m$}
\end{enumerate}
a.e.~in $\mathcal{O}$, let $A^\pm\in M(\alpha,\beta,\Omega^\pm)$ and  $A^\lay_{\eps}\in M(\alpha,\beta,\OmEps^\lay)$. Moreover, we assume $A^\lay \in C^0(\overline{\omega},L^{\infty}(Y_0))^{3\times 3 \times 3 \times 3}$ and $Y'$-periodic with respect to the second variable.

\item We assume $\varrho^{\pm} \in L^{\infty}(\Omega^{\pm})$ and $\varrho^\lay \in C^0(\overline{\omega},L^{\infty}(Y_0))$ such that $\varrho^\lay$ is $Y'$-periodic with respect to the second variable. Further, we assume
\begin{align*}
0<\varrho_0<\varrho^\pm(x)<\varrho_1 \qquad \text{ for a.e. } x\in\Omega^\pm
\end{align*}
and
\begin{align*}
0<\varrho_0<\varrho^\lay(x',y)<\varrho_1 \qquad \text{ for a.e. } (x',y) \in \omega \times Y_0,
\end{align*}
for some $\varrho_0,\varrho_1\in\R_+$.

\item For the boundary force $g$, we assume $g\in H^1(S,L^2(\GamN))^3$ and $g =  0$ on $\left(-\frac{\delta}{2},\frac{\delta}{2}\right)\times \partial \omega$ for some $\delta >\eps$.

\item The body forces fulfil  $f^\pm\in L^2(S\times\Omega^\pm)^3$ and  $f^\lay_\eps\in L^2(S\times\OmEps^\lay)^3$ such that
\begin{equation}\label{est:fM}
    \frac{1}{\sqrt{\eps}}\norm{f^\lay_\eps}_{L^2(S\times\OmEps^\lay)^3}\leq C
\end{equation}
for a constant $C$ independent of $\eps$. Further, we assume that there exists $f^\lay \in L^2(S \times \omega)^3$ such that $\TEps(f_{\eps}^\lay) \weak f^\lay$ in $L^2(S \times \omega \times Y_0)^3$, where $\TEps $ is the unfolding operator defined in Definition \ref{def:unfolding_operator}. 

\item For the initial condition $u^{\eps}_0$, we assume $u^\pm_0\in H^1_{\GamD\cap\partial\Omega^\pm}(\Omega^\pm)$ and $u^\lay_{0,\eps}\in H^1_{\GamD\cap\partial\OmEps^\lay}(\OmEps^\lay)$ with 
\begin{equation}\label{est:uM0}
    \frac{1}{\sqrt{\eps}}\norm{u_{0,\eps}^\lay}_{L^2(\OmEps^\lay)^3}+\eps^{\frac{\gamma}{2}}\norm{e(u_{0,\eps}^\lay)}_{L^2(\OmEps^\lay)^{3\times 3}}\leq C 
\end{equation}
for a constant $C$ independent of $\eps$, and there exists $u_0^\lay \in L^2(\omega, \widetilde{H}^1_{\#}(Y_0))^3$ (for the definition of $\widetilde{H}^1_{\#}(Y_0)$, see \eqref{def:H1_tilde_per_Y0}) such that $\TEps(u_{0,\eps}^\lay) \weak u^\lay_0$ in $L^2(\omega \times Y_0)^3$.

\item The initial condition $u_1^{\eps}$ fulfils  $u^\pm_1\in L^2(\Omega^\pm)^3$ and $u^\lay_{1,\eps}\in L^2(\OmEps^\lay)^3$ with
\begin{equation}\label{est:uM1}
    \frac{1}{\sqrt{\eps}}\norm{u^\lay_{1,\eps}}_{L^2(\OmEps^\lay)^3}\leq C
\end{equation}
for a constant $C$ independent of $\eps$. Further, we assume that there exists $u_1^\lay \in L^2(\omega \times Y_0)^3$ such that $\TEps(u_{1,\eps}^\lay) \weak u_1^\lay$ in $L^2(\omega \times Y_0)^3$.

\item\label{ass:zero_Dirichlet_BC_gamma-1-3} In case $\gamma = -1 $ or $\gamma =-3$, we additionally assume homogeneous Dirichlet boundary conditions on the lateral part of the thin layer:
\begin{align*}
    \ueps = 0 \qquad\mbox{ on } S\times \left( \left(-\frac{\eps}{2},\frac{\eps}{2}\right)  \times \partial \omega \right).
\end{align*}
\end{enumerate}

\begin{remark}\label{rem:Assumptions} \ 
\begin{enumerate}[label = (\roman*)]
\item It is implicit in the above that $\eps$ is not only assumed much smaller than the characteristic macroscopic length, but also that $\eps$ is much smaller than any wavelength arising in the problem.

\item The assumption $g = 0$ at the lateral boundary of the thin layer is for technical reasons to avoid additional boundary terms in the thin layer. Under suitable assumptions, we expect that the inhomogeneous case can be treated along the same lines, but more details would have to be worked out.

\item We stress the crucial difference in the boundary condition in assumption \ref{ass:zero_Dirichlet_BC_gamma-1-3} compared to the case $\gamma  =1$. From a mathematical point of view, the reason is given by the different Korn inequalities used for the derivation of the a priori estimates, see Section \ref{sec:existence_apriori_estimates}.

\item It is straightforward also to consider $f^\lay \in L^2(S\times \omega \times Y_0)^3$. In fact, it is relevant to have $f^\lay$ independent of $y$ only for the representation of $u^\lay$ (see Proposition \ref{prop:decomposition_u^M}) to obtain a cell problem for the force term independent on $f^\lay$. 

\item For $\gamma = -1,-3$, we have from the estimate $\eqref{est:uM0}$ that $u_0^\lay \in L^2(\omega)^3$ is independent of $y$, and for $\gamma = -3$ we even have $[u_0^\lay]_2 = [u_0^\lay]_3= 0$. For more details, see the compactness results in Section \ref{sec:compactness}.
\end{enumerate}

\end{remark}

\section{Main results and the macroscopic model}\label{sec:mainresults}

In this section, we formulate the macroscopic models obtained in the limit $\eps \to 0$ and the main results of the paper, which are proven in the subsequent sections. We have to distinguish between the three cases $\gamma = 1,-1,-3$. Of course, the main difference lies in the interface conditions across $\omega$. More precisely, in all cases we have the existence of the macroscopic bulk displacements $u^{\pm}$ such that the convergence 
\begin{align*}
    \chi_{\OmEps^{\pm} } \ueps \weak u^{\pm} \qquad\mbox{weakly in } L^2(S \times \Omega^{\pm})^3
\end{align*}
holds and we refer to Proposition \ref{prop:compactness_bulk_domains} for the precise compactness results for the sequence $\ueps$ in the bulk domains. The displacements $u^{\pm}$ solve again the wave equations with the same coefficients as in the microscopic bulk domains. 
The critical question is the interface conditions between $u^+$ and $u^-$ across $\omega$. We will see that we obtain continuity in the cases $\gamma = -1$ and $\gamma = -3$, but in the case $\gamma = 1$ the thin layer acts macroscopically as a thick interface and we have a jump between the displacements $u^+$ and $u^-$. The interface conditions on $\omega$ depend on the limit functions for the displacement $\ueps$ in the perforated layer $\OmEps^\lay$, which strongly depend on $\eps$. Here, we have to take into account the effect of homogenisation (related to the periodic perforations) and, simultaneously, of the dimension reduction (related to the vanishing thickness). To obtain suitable compactness results, we use the unfolding operator in thin domains, which gives a characterisation for the two-scale convergence in thin heterogeneous layers (see \cite{Neuss:2007}). We refer to Section \ref{sec:compactness} for the precise compactness results.  Now, let us formulate the macroscopic models for the three different cases:

\subsection{The case $\gamma = 1$}\label{sec:mainresults_gam1}

Different to the other cases, the limit displacement $u^\lay$ in the thin layer also depends on the microscopic variable $y$ in the case $\gamma =1$. More precisely, we have that 
$u^\lay \in L^{\infty}(S;L^2(\omega,\widetilde{H}_{\#}^1(Y_0)))^3 $ (see $\eqref{def:H1_tilde_per_Y0}$ for the definition of the space $\widetilde{H}_{\#}^1(Y_0)$) and the convergence
\begin{align*}
    \TEps(\ueps) \weak u^\lay \qquad\mbox{ weakly in } L^2(S \times \omega \times Y_0)^3
\end{align*}
holds.
For the definition of the unfolding operator, we refer to Definition \ref{def:unfolding_operator}.
On the interface $\omega$, we have the interface condition 
\begin{align*}
u^{\pm}|_{\omega} = u^\lay \qquad\mbox{on } S\times \omega \times S^{\pm};
\end{align*}
in particular, we expect a jump of the bulk displacements $u^+$ and $u^-$ across $\omega$. We emphasise that the limit function $u^\lay$ is constant with respect to $y$ on $S^{\pm}$.

The triple $(u^+,u^\lay,u^-)$ is the unique weak solution  to the problem
\begin{subequations}\label{eq:macro_model_gamma1_strong}
\begin{align}
\label{eq:two_scale_model_gamma1_PDE_pm}
\partial_t(\varrho^\pm\partial_t u^\pm)-\nabla\cdot(A^\pm e(u^\pm))&=\varrho^\pm f^\pm &&\text{ in }S\times\Omega^\pm,
\\
\label{eq:two_scale_model_gamma1_PDE_M}
\partial_t(\varrho^\lay\partial_t u^\lay)-\nabla\cdot(A^\lay e_y(u^\lay))&=\varrho^\lay f^\lay &&\text{ in }S\times \omega \times Y_0,
\end{align}
together with the boundary conditions
\begin{align}
u^\pm &= 0 &\text{ on }S\times(\GamD\cap\partial\Omega^\pm),\\
		(A^\pm e(u^\pm))\nu &= g &\text{ on }S\times(\GamN\cap\partial\Omega^\pm),
\end{align}
the interface conditions 
\begin{align}
\label{eq:macro_model_gamma1_interface_displacement}
u^\lay &= u^{\pm} &\mbox{ on }& S \times \omega \times S_Y^{\pm},
\\
\label{eq:macro_model_gamma1_interface_stress}
\int_{S_Y^{\pm}} A^\lay e_y(u^\lay) \Sy \nu &= A^{\pm} e(u^{\pm}) \nu &\mbox{ on }& S \times \omega,
\end{align}
as well as the initial conditions
\begin{align}
u^{\pm}(0) &= u_0^{\pm} &\mbox{ in }& \Omega^{\pm},
\\
\partial_t u^{\pm}(0) &= u_1^{\pm} &\mbox{ in }& \Omega^{\pm},
\\
u^\lay (0) &= u_0^\lay &\mbox{ in }& \omega \times Y_0,
\\
\partial_t u^\lay(0) &= u_1^\lay &\mbox{ in  }& \omega \times Y_0.
\end{align}
\end{subequations}
We see that, on the macroscopic scale, the thin layer acts as a thick interface between $\Omega^+$ and $\Omega^-$, the behaviour of which is given by the equation for the displacement $u^\lay$. This interface encapsulates the heterogeneity of the thin layer in the micromodel.  Let us formulate the definition of a weak solution of this system.  Defining
\begin{multline}\label{def:spaceH_1}
\spaceH_1:= \bigg\{(v^+,v^\lay,v^-) \in H^1_{\GamD\cap\partial\Omega^+}(\Omega^+)^3  \times  L^2(\omega,\widetilde{H}^1_{\#}(Y_0))^3 \times H^1_{\GamD\cap\partial\Omega^-}(\Omega^-)^3  \, | \, \\ v^{\pm} = v^\lay \mbox{ on } \omega \times S_Y^{\pm}  \bigg\},
\end{multline}
we call $(u^+,u^\lay,u^-)$ a weak solution of the macroscopic problem $\eqref{eq:macro_model_gamma1_strong}$ if $(u^+,u^\lay,u^-)\in L^2(S,\spaceH_1)$,
$
 \partial_t (u^+,u^\lay,u^-) \in L^2(S;L^2(\Omega^+)^3\times L^2(\omega \times Y_0)^3 \times L^2(\Omega^-)^3) 
$ 
and $(u^+,u^\lay,u^-)(0) = (u_0^+,u_0^\lay,u_0^-)$ as well as
\begin{align}
\begin{aligned}\label{eq:macro_model_two_scale_gamma_1}
	-&\sum_{\pm}\int_0^T\int_{\Omega^\pm} \left\{\varrho^\pm\partial_t u^\pm\cdot\partial_t v^\pm\,\x\t -\int_0^T\int_{\Omega^\pm} A^\pm e(u^\pm):e(v^\pm)\,\x\t \right\}\\
    &-\int_0^T\int_{\omega\times Y_0} \varrho^\lay \partial_t u^\lay\cdot\partial_t v^\lay\,\x'\y\t +\int_0^T\int_{\omega\times Y_0} A^\lay e_y(u^\lay): e_y(v^\lay)\,\x'\y\t\\
	=& \sum_{\pm}\int_0^T\int_{\Omega^\pm} \varrho^\pm f^\pm\cdot v^\pm\,\x\t +\int_0^T\int_{\GamN\cap\partial\Omega^\pm} g\cdot v^\pm\,\mathrm{d}\sigma(x)\t +\int_{\Omega^\pm}\varrho^\pm u^\pm_1\cdot v^\pm(0,x)\,\x\\
    &+\int_0^T\int_{\omega\times Y_0} \varrho^\lay f^\lay\cdot v^\lay\,\x'\y\t +\int_{\omega\times Y_0} \varrho^\lay u^\lay_1\cdot v^\lay(0,x',y)\,\x'\y
\end{aligned}
\end{align}
for all $v:=(v^+,v^\lay,v^-) \in L^2(S,\spaceH_1)$ with $v(T) = 0$,
$\partial_t v^\pm\in L^2(S\times\Omega^\pm)^3$
and $\partial_tv^\lay\in L^2(S\times\omega\times Y_0)^3$.

Let us give some comments on the regularity of the solution and also give a stronger weak formulation.
The equation $\eqref{eq:two_scale_model_gamma1_PDE_pm}$ is valid in the space $L^2(S,H^{-1}(\Omega^{\pm}))^3$. This can be seen by choosing in equation $\eqref{eq:macro_model_two_scale_gamma_1}$ test functions of the form $(\phi^+,0,0) \in C_0^{\infty}(S,\spaceH_1)$ for $\phi^+ \in C_0^{\infty}(S,H^1_0(\Omega^+))^3$, and similarly with $\phi^-$. In particular, we get $ \partial_t(\varrho^{\pm}\partial_t u^{\pm}) \in L^2(S,H^{-1}(\Omega^{\pm}))^3$ and, therefore, $(\varrho^{\pm} \partial_tu^{\pm})(0) \in H^{-1}(\Omega^{\pm})^3$  and equal to $\varrho^{\pm} u_1^{\pm}$. In a slight abuse of notation, we simply write $\partial_t u^{\pm}(0) = u_1^{\pm}$. However, under the condition $\varrho^{\pm} \in W^{1,\infty}(\Omega^{\pm})$, this equation is also valid in $L^2(\Omega^{\pm})^3.$ In a similar way, we can proceed for the function $u^\lay$ by choosing $(0,v^\lay,0) \in C_0^{\infty}(S,\spaceH_1)$ such that $v^\lay = 0$ on $S_Y^{\pm}$, which implies $\partial_t (\varrho^\lay \partial_t u^\lay) \in L^2(S,L^2(\omega,H^{-1}(Y_0)))^3$. 
Further, since all the terms in $\eqref{eq:macro_model_two_scale_gamma_1}$ which include no time derivatives define linear functionals on $\spaceH_1$, we also obtain $\partial_t (\varrho^+  \partial_t u^+, \varrho^\lay  \partial_t u^\lay, \varrho^- \partial_t u^-) \in L^2(S,\spaceH_1')^3$.
We summarise these observations in the following remark:
\begin{remark}
The solution to \eqref{eq:macro_model_two_scale_gamma_1} satisfies
\begin{align*}
      \partial_t(\varrho^{\pm}\partial_t u^{\pm}) \in L^2(S,H^{-1}(\Omega^{\pm}))^3 \quad \mbox{ and } \quad \partial_t (\varrho^\lay \partial_t u^\lay) \in L^2(S,L^2(\omega,H^{-1}(Y_0)))^3
\end{align*}
as well as $\varrho^{\pm} \partial_t u^{\pm}(0) = \varrho^{\pm} u_1^{\pm}$ and $\varrho^\lay \partial_t u^\lay(0) = \varrho^\lay u^\lay_1$. Additionally, we have
\begin{align*}
    \partial_t (\varrho^+  \partial_t u^+, \varrho^\lay  \partial_t u^\lay, \varrho^- \partial_t u^-) \in L^2(S,\spaceH_1')^3.
\end{align*}
Therefore, we can write equation $\eqref{eq:macro_model_two_scale_gamma_1}$ in the following equivalent form: For all $v \in \spaceH$, it holds almost everywhere in $S$ that 
\begin{align}
\begin{aligned}\label{eq:two_scale_macro_gamma1_aux}
\big\langle &\partial_t (\varrho^+  \partial_t u^+, \varrho^\lay  \partial_t u^\lay, \varrho^- \partial_t u^-), v \big\rangle_{\spaceH_1',\spaceH} +\sum_{\pm}\int_{\Omega^\pm} A^\pm e(u^\pm):e(v^\pm)\,\x\\
    &+\int_{\omega\times Y_0} A^\lay e_y(u^\lay): e_y(v^\lay)\,\x'\y\\
	=& \sum_{\pm}\int_{\Omega^\pm} \varrho^\pm f^\pm\cdot v^\pm\,\x + \int_{\GamN\cap\partial\Omega^\pm} g\cdot v^\pm\,\mathrm{d}\sigma(x)  +\int_{\omega\times Y_0} \varrho^\lay(y) f^\lay\cdot v^\lay\,\x'\y.
\end{aligned}
\end{align}
\end{remark}
In general, we cannot expect the macroscopic displacements $u^+$ and $u^-$ to be continuous across $\omega$, i.e.~we have $u^+ \neq u^-$ on $\omega$, and similarly for the normal stress. Similarly as in \cite{Neuss:2007}, we formulate a condition for the jump of the displacement and the normal stress. More precisely, the jump of the displacement  on the interface $\omega$ can be computed by
\begin{align*}
 u^+_k(t,0,x')-u^-_k(t,0,x') = \int_{Y_0} A^\lay e_y(\eta^{(k)}) : e_y(u^\lay) \,\y 
\end{align*}  
    for $k\in\{1,2,3\}$ and,  for the stress, we have in the distributional sense
\begin{align*}
 (A^{+}e(u^{+})n - A^{-}e(u^{-})n)(t,0,x') = \int_{Y_0} (\varrho^\lay f^\lay - \partial_t(\varrho^\lay\partial_t u^\lay))\,\y,
\end{align*}
    where $n$ is normal from $\Omega^+$ to $\Omega^-$, i.e.~$n=(-1,0,0)$, and the functions $\eta^{(k)}$ are solutions to the cell problems \eqref{cell_problem_eta^k_jump}. 

Finally, in applications, the macroscopic displacements $u^+$ and $u^-$ are typically of particular interest. However, for the calculation of these quantities we have to solve the full problem for $(u^+,u^\lay,u^-)$ coupled via the interface conditions $\eqref{eq:macro_model_gamma1_interface_displacement}$ and $\eqref{eq:macro_model_gamma1_interface_stress}$. In what follows, we formulate a version of the macroscopic model for the displacements $u^+$ and $u^-$ with effective interface conditions across $\omega$ not dependent (explicitly) on $u^\lay$. For this, we make the additional assumptions $u^\lay_0 = u^\lay_1 = 0$ (only to simplify some calculations) and $f^\lay = 0$, and further assume slightly higher regularity for $\partial_t u^{\pm}|_{\omega}$, see Section \ref{sec:effective_interface_conditions} for more details.  Following the ideas in \cite{gahn2025effective} and \cite{gahn2025b_effective}, we first derive a representation for $u^\lay$ via suitable cell problems (see Proposition \ref{prop:decomposition_u^M})
, use this representation in the weak formulation $\eqref{eq:two_scale_macro_gamma1_aux}$ and choose specific forms of the test function $v^\lay$ to obtain an equation independent of $u^\lay$ including effective coefficients depending on the cell problems. Finally, we see that $(u^+,u^-)$ is a weak solution to the problem 
\begin{align}
\begin{aligned}\label{eq:macro_model_gamma1_effective}
\partial_t (\varrho^{\pm} \partial_t u^{\pm}) - \nabla \cdot (A^{\pm} e(u^{\pm})) &= \varrho^{\pm} f^{\pm} &\mbox{ in }& S\times \Omega^{\pm},
\\
u^{\pm} &= 0 &\mbox{ on }& S \times \Gamma_D\cap\partial \Omega^{\pm},
\\
A^{\pm}e(u^{\pm})\nu &= g &\mbox{ on }& S\times \Gamma_N,
\\
A^{\pm}e(u^{\pm})\nu &=  H^{\pm}(t,x',u^+,u^-) &\mbox{ on }& S \times \omega,
\\
u^{\pm}(0) = \partial_t u^{\pm}(0) &= 0 &\mbox{ in }& \Omega^{\pm}.
\end{aligned}
\end{align}
Here, the right-hand side in the effective interfacial stress is given by
\begin{align}\label{eq:Hbeta}
 H^{\beta}(t,x',u^+,u^-) = \sum_{\alpha \in\{\pm\}}  \int_0^t G^{\alpha\beta}(t-s,t,x')\partial_t u^{\alpha}(s,0,x') + F^{\alpha\beta}(t-s,t,x') \partial_{tt}u^{\alpha}(s,0,x') \s  
\end{align}
with effective coefficients 
\begin{align*}
    G_{ji}^{\alpha\beta} (\tau,t,x') &:= \int_{S_Y^{\beta}} A^\lay e_y(\chi_i^{\alpha})(\tau,x',y) \nu \cdot e_j \,\mathrm{d}\sigma(y),
\\
F_{ji}^{\alpha\beta}(\tau,t,x') &:= \int_{S_Y^{\beta}} A^\lay e_y(\eta_i^{\alpha})(\tau,x',y) \nu \cdot e_j \,\mathrm{d}\sigma(y),
\end{align*}
$i,j=1,2,3$ and $\alpha,\beta \in \{\pm\}$, 
where the cell solutions $\chi_i^{\pm}$ and $\eta_i^{\pm}$ for $i=1,2,3$ are given by  $\eqref{eq:cell_problem_chi^pm}$ and $\eqref{eq:cell_problem_eta^pm}$, respectively. Another formulation, using less regularity for the cell solutions, for $G^{\alpha\beta}$ and $F^{\alpha\beta}$ is stated in $\eqref{def:effective_G}$ and $\eqref{def:effecitve_F}$.
The effective interfacial stress  \eqref{eq:Hbeta} constitutes a memory term taking into account the delay of force transmission across the interface.

\subsection{The case $\gamma = -1$}\label{sec:mainresults_gam-1}

For $\gamma = -1$, we obtain a macroscopic displacement $u^\lay$ on the interface $\omega$ in the limit, such that $u^\lay \in L^2(S;H_0^1(\omega)^2 \times L^2(\omega)) $ and  
\begin{align*}
    \TEps(\ueps) \weak u^\lay \qquad\mbox{ weakly in } L^2(S \times \omega \times Y_0)^3.
\end{align*}
In particular, the limit function in the interface is independent of the microscopic variable $y$ and describes a macroscopic quantity.
Further, we have the continuity condition for the bulk displacements
\begin{align*}
    u^+ = u^\lay = u^- \qquad \mbox{on } S \times \omega.
\end{align*}
Now, the critical aspect is the stress condition for $u^+$ and $u^-$ across $\omega$. Let us formulate the full macroscopic model: The triple $(u^+,u^\lay,u^-)$ is the unique weak solution to the problem
\begin{subequations}\label{Macro_Model_gamma_-1}
\begin{align}
\partial_t (\varrho^{\pm} \partial_t u^{\pm}) - \nabla\cdot (A^{\pm} e(u^{\pm})) &= \varrho^{\pm} f^{\pm} &\mbox{ in }& S \times \Omega^{\pm},
\\
(A^{\pm} e(u^{\pm}))\nu &= g &\mbox{ on }& S \times \Gamma_N,
\\
u^{\pm} &= 0 &\mbox{ on }& S \times \Gamma_D,
\\
u^{\pm} &= u^\lay &\mbox{ on }& S \times \omega,
\\
-\llbracket A^{\pm} e(u^{\pm})\nu^{\pm}\rrbracket &= \partial_t( \bar{\varrho}^\lay \partial_t u^\lay) &\mbox{ }& \notag
\\ \label{Macro_Model_gamma_-1_interf}
&\quad - \nabla_{x'} \cdot (A^{\ast} e_{x'}(\hat{u}^\lay)) - \bar{\varrho}^\lay f^\lay &\mbox{ on }& S\times \omega,
\\
u^\lay &= 0 &\mbox{ on }& S \times \partial \omega,
\\
u^{\pm}(0) &= u_0^\lay &\mbox{ in }& \Omega^{\pm},
\\
\partial_t u^{\pm}(0) &= u_1^{\pm} &\mbox{ in }& \Omega^{\pm},
\\
u^\lay(0)  &= u_0^\lay &\mbox{ in }& \omega,
\\
\partial_t u^\lay (0) &=  u_1^\lay  &\mbox{ in }& \omega,
\end{align}
\end{subequations}
with $\hat{u}^\lay = ([u^\lay]_2,[u^\lay]_3)$, where $[\,\cdot\,]_j$ denotes the $j$th component of the vector, and the homogenised elasticity tensor $A^{\ast}$ defined in $\eqref{def:effective_elasticity_tensor}$. Furthermore, the jump of the normal stress is defined by
\begin{align*}
    \llbracket A^{\pm} e(u^{\pm})\nu^{\pm} \rrbracket  := A^+ e(u^+) \nu^+  + A^- e(u^-) \nu^- ,
\end{align*}
where $\nu^{\pm}$ denotes the outer unit normal with respect to $\Omega^{\pm}$, and $\bar{\varrho}^\lay$ is the mean value given by 
\begin{align*}
    \bar{\varrho}^\lay:= \int_{Y_0} \varrho^\lay \y.
\end{align*}
We emphasise that the spatial differential operator on the right-hand side of the jump condition for the normal stress (equation \eqref{Macro_Model_gamma_-1_interf}) acts only on the second and third components of $u^\lay$.

Let us give the definition of a weak solution of the macromodel. Introducing the solution space
\begin{equation}\label{def:spaceH_-1}
\spaceH_{-1}:= \left\{\phi \in H_{\Gamma_D}^1(\Omega)^3 \, : \, (\phi_1,\phi_2)|_{\omega} \in H_0^1(\omega)^2 \right\},
\end{equation}
we call the triple $(u^+,u^\lay,u^-)$ a weak solution to the macroscopic model $\eqref{Macro_Model_gamma_-1}$ if 
\begin{align*}
    u^{\pm} &\in L^2(S;H^1_{\Gamma_D}(\Omega^{\pm}))^3 \cap H^1(S; L^2(\Omega^{\pm}))^3,
    \\
    u^\lay &\in L^2(S;H_0^1(\omega)^2 \times L^2(\omega)) \cap H^1(S; L^2(\omega))^3
\end{align*}
 with  $(u^+,u^\lay,u^-)(0) = (u^+_0,u^\lay_0,u^-_0)$ and   such that for all $v \in L^2(S;\spaceH_{-1})$ with $\partial_t v \in L^2(S;L^2(\Omega))^3$ and $v(T) = 0$ it holds that
 \begin{align*}
    -&\sum_{\pm}\int_0^T\int_{\Omega^{\pm}} \varrho^{\pm}\partial_t u^{\pm}\cdot\partial_t v^{\pm}\,\x\t
    -\int_0^T\int_{\omega\times Y_0} \varrho^\lay \partial_t u^\lay\cdot\partial_t v^\lay\,\x'\y\t 
    \\
    &+\sum_{\pm}\int_0^T\int_{\Omega^{\pm}} A^{\pm}e(u^{\pm}):e(v^{\pm})\,\x\t +\int_0^T\int_{\omega} A^{\ast} e_{x'} (\hat{u}) : e_{x'}(\hat{v}^\lay) \x'\t\\
	=& \sum_{\pm}\int_0^T\int_{\Omega^{\pm}} \varrho^{\pm}f^{\pm}\cdot v^{\pm}\,\x\t 
    +\int_0^T\int_{\omega\times Y_0} \varrho^\lay(y) f^\lay\cdot v^\lay\,\x'\y\t\\
    &+ \sum_{\pm}\int_0^T\int_{\GamN\cap\partial\Omega^{\pm}} g\cdot v^{\pm}\,\mathrm{d}\sigma(x)\t 
    +\sum_{\pm}\int_{\Omega^{\pm}}\varrho^{\pm}u^{\pm}_1\cdot v^{\pm}(0,x)\,\x\\
    &+\int_{\omega\times Y_0} \varrho^\lay u^\lay_1(x',y)\cdot v(0,x')\,\x'\y,
 \end{align*}
where $\hat{u}^\lay=([u^\lay]_2,[u^\lay]_3)$, $v^\lay:= v^{\pm}|_{\omega}$ and  $\hat{v}^\lay = ([v^\lay]_2,[v^\lay]_3).$ Note that the limit function $u := u^{\pm}$ in $\Omega^{\pm}$ is an element of $L^2(S;\spaceH_{-1}).$

\subsection{The case $\gamma  = -3$}\label{sec:mainresults_gam-3}

For an even stiffer layer, the behaviour in the bulk domains is comparable to the previous case $\gamma = -1$; in particular, the continuity of the bulk displacements $u^+$ and $u^-$ across $\omega$ also holds. However, the structure of the limit function in the interface changes. In fact, for $\gamma  =-3$, the interface limit function behaves like a Kirchhoff--Love displacement and we have the existence of limit functions $[u^\lay]_1 \in L^2(S;H_0^2(\omega))$ and $\hat{u}^\lay \in L^2(S;H_0^1(\omega))^2$ such that 
\begin{align*}
  \TEps(\ueps)_1 &\weak [u^\lay]_1 &\mbox{ weakly in } L^2(S\times \omega \times  Y_0),
  \\
  \frac{1}{\eps}\TEps(\ueps)_{\alpha} &\weak [\hat{u}^\lay]_{\alpha} &\mbox{ weakly in } L^2(S\times \omega \times Y_0), 
\end{align*}
$\alpha = 2,3$.
Here, we see that $[u^\lay]_1$ is the zeroth-order approximation (in the two-scale sense) of the first component of $\ueps|_{\OmEps^\lay}$, and $\hat{u}^\lay$ is the first-order approximation of $([\ueps]_2,[\ueps]_3)|_{\OmEps}$ (the zeroth-order approximation vanishes). We use the notation $[u^\lay]_1$ to allude to the idea of viewing $u^\lay$ as a three-dimensional vector field with second and third component equal to zero.

As already mentioned above, we have the continuity of the bulk displacements across $\omega$. More precisely, we have
\begin{align*}
    u^{\pm}|_{\omega} = [u^\lay]_1 e_1 \qquad\mbox{on } S\times \omega
\end{align*}
with the unit vector $e_1$. In particular, the displacement in the tangential direction is zero.

Next, we formulate the macroscopic model.  The limit function $(u^+,[u^\lay]_1,\hat{u}^\lay,u^-)$ is the solution to
\begin{subequations}\label{Macro_Model_gamma_-3}
\begin{align}
\partial_t (\varrho^{\pm} \partial_t u^{\pm}) - \nabla\cdot (A^{\pm} e(u^{\pm})) &= \varrho^{\pm} f^{\pm} &\mbox{ in }& S \times \Omega^{\pm},
\\
(A^{\pm} e(u^{\pm}))\nu &= g &\mbox{ on }& S \times \Gamma_N,
\\
u^{\pm} &= 0 &\mbox{ on }& S \times \Gamma_D,
\\
u^{\pm} &= [u^\lay]_1 e_1 &\mbox{ on }& S \times \omega,
\\
-\llbracket A^{\pm} e(u^{\pm})\nu^{\pm} \cdot \nu^{\pm}\rrbracket &= \partial_t( \bar{\varrho}^\lay \partial_t [u^\lay]_1) - \bar{\varrho}^\lay f^\lay_1  &\mbox{ }& \notag
\\
&\quad + \nabla^2_{x'} : (b^{\ast} e_{x'}(\hat{u}^\lay) 
\\
&\quad + c^{\ast} \nabla_{x'}^2 [u^\lay]_1) &\mbox{ on }& S\times \omega, \notag
\\
-\nabla_{x'} \cdot (a^{\ast} e_{x'}(\hat{u}^\lay) + b^{\ast} \nabla_{x'}^2 [u^\lay]_1 ) &= 0 &\mbox{ in }& S\times \omega,
\\
[u^\lay]_1 &= 0 &\mbox{ on }& S \times \partial \omega,
\\
\nabla_{x'} [u^\lay]_1 &= 0 &\mbox{ on }& S\times \partial \omega,
\\
\hat{u}^\lay &= 0 &\mbox{ on }& S\times \partial \omega,
\\
u^{\pm}(0) &= u_0^{\pm} &\mbox{ in }& \Omega^{\pm},
\\
\partial_t u^{\pm}(0) &= u_1^{\pm} &\mbox{ in }& \Omega^{\pm},
\\
[u^\lay]_1(0)  &= [u_0^\lay]_1 &\mbox{ in }& \omega,
\end{align}
\end{subequations}
where the effective coefficients are given in \eqref{def:effective_coefficients_plate}.
Compared to the case $\gamma = -1$, we see that we only get a jump condition for the normal component of the normal stress. Nevertheless, we still have enough interface conditions to ensure the problem is well-defined since we have the vanishing Dirichlet condition on $\omega$ for the second and third component of the displacement $u^{\pm}$. 

In order to state the weak formulation of this problem, we introduce the corresponding solution space
\begin{equation}\label{def:spaceH_-3}
\spaceH_{-3}:= \left\{ \phi \in H^1(\Omega)^3 \, : \, \phi = 0 \mbox{ on } \Gamma_D, \, \, \phi|_{\omega} = (\phi_3,0,0)|_{\omega} \in H_0^2(\omega)^3 \right\}.
\end{equation}
Then, $(u^+,[u^\lay]_1,\hat{u}^\lay,u^-)$ is a weak solution of the macroscopic problem $\eqref{Macro_Model_gamma_-3}$ if 
\begin{align*}
    u^{\pm} &\in L^2(S;H_{\Gamma_D}^1(\Omega^{\pm}))^3 \cap H^1(S;L^2(\Omega^{\pm}))^3,
    \\
    [u^\lay]_1 &\in L^2(S;H_0^2(\omega))\cap H^1(S;L^2(\omega)),
    \\
    \hat{u}^\lay &\in L^2(S;H_0^1(\omega))^2,
\end{align*}
with $u^{\pm}(0) = u_0^{\pm}$ and $[u^\lay]_1 (0) = [u_0^\lay]_1$, and for all $V \in L^2(S;\spaceH_{-3})$ such that $\partial_t V \in L^2(S \times \Omega)^3$ and $\bar{U} := (U_2,U_3) \in L^2(S;H_0^1(\omega))^2$  with $\partial_t \bar{U} \in  L^2(S;L^2(\omega))^2$ and $V(T) = 0$ and $\bar{U}(T) = 0$ there holds 
\begin{align*}
- \sum_{\pm}& \left\{ \int_0^T \int_{\Omega^{\pm}} \varrho^{\pm} \partial_t u^{\pm} \cdot \partial_t V \x \t  - \int_0^T \int_{\Omega^{\pm}} A^{\pm} e(u^{\pm}) : e(V) \x \t \right\}
\\
-& \int_0^T \int_{\omega} \int_{Y_0} \varrho^\lay \partial_t [u^\lay]_1 \partial_t V_1 \y\x'\t \\
+& \int_0^T \int_{\omega} a^{\ast} e_{x'}(\hat{u}^\lay)  : e_{x'}(\bar{U}) + b^{\ast} \nabla_{x'}^2 [u^\lay]_1 : e_{x'}(\bar{U}) + b^{\ast} e_{x'} (\hat{u}^\lay) : \nabla_{x'}^2 V_1 + c^{\ast} \nabla_{x'}^2 [u^\lay]_1 : \nabla_{x'}^2 V_1\x'\t
\\
&= \sum_{\pm} \left\{ \int_0^T \int_{\Omega^{\pm}} \varrho^{\pm} f^{\pm}   \cdot V \x \t  + \int_{\Omega^{\pm}} \varrho^{\pm} u_1^{\pm} \cdot V(0) \x + \int_0^T \int_{\Gamma_N} g \cdot V \mathrm{d}\sigma(x) \t\right\}
\\
&\hspace{3em} + \int_0^T \int_{\omega} \int_{Y_0} \varrho^\lay f_1^\lay V_1 \y \x' \t + \int_{\omega} \int_{Y_0} \varrho^\lay [u^\lay_1]_1 V_1 \y \x'.
\end{align*}

\section{Existence and uniform boundedness}\label{sec:existence_boundedness}
\label{sec:existence_apriori_estimates}

While for fixed $\eps$ the existence and uniqueness of a weak solution of the microscopic model in the sense of Definition \ref{def:weak_micro_solution} is standard, the crucial point is to establish a priori estimates uniform with respect to $\eps$. An essential ingredient is the Korn inequality, where the critical point is to know the explicit dependence of the Korn constant (which depends on the domain $\OmEps$) on the parameter $\eps$. Here, we have to distinguish between the cases $\gamma = 1$ and $\gamma = -1,-3$. 
For the case $\gamma  =1$, we use the following Korn-type inequality from \cite[Proposition 3.4]{Griso:2015}.

\begin{lemma}[Korn inequality I]\label{prop:estimate_H1_OmEps}
There exists a constant $C$ independent of $\eps$ such that for every $u\in\Hgam(\OmEps)^3$
\[ \norm{u}_{H^1(\OmEps)^3} \leq C\norm{e(u)}_{L^2(\OmEps)^{3\times 3}}. \]
Moreover, the following estimate holds:
\begin{equation}\label{ineq:estimate_H1_OmEps}
    \begin{split}
    \sum_{\pm} \norm{u}_{H^1(\OmEps^{\pm})}&+\frac{1}{\sqrt{\eps}}\norm{u}_{L^2(\OmEps^\lay)}+\sqrt{\eps}\norm{\nabla u}_{L^2(\OmEps^\lay)}\\
    &\leq C\left( \sum_{\pm} \norm{e(u)}_{L^2(\OmEps^{\pm})}+\sqrt{\eps}\norm{e(u)}_{L^2(\OmEps^\lay)} \right).
\end{split}
\end{equation}    
\end{lemma}
This allows us to control the $H^1$-norm of the displacement (in the bulk domains and the thin layer) by the symmetrised gradient in the bulk domains and a "small" part (scaled with some order of $\eps$) in the thin layer. 

In the cases $\gamma = -1,-3$, we use a Korn inequality for perforated domains which allows to control the $H^1$-norm of the displacement in the layer only by the symmetrised gradient in the layer. The price is an additional assumption (vanishing Dirichlet boundary condition) on the lateral boundary $\partial \omega$ as well as different scalings for the different components of the displacement, see  \cite[Theorem 2]{gahn2021two} or \cite[Theorem 6]{griso2020homogenization}.

\begin{lemma}[Korn inequality II]\label{lem:KornII}
For every $\weps  \in H^1(\OmEps^\lay)^3$ with $\weps = 0$ on $(-\eps/2 , \eps/2) \times \partial \omega $, the following inequality holds: 
\begin{align*}
\sum_{i=2}^3 \foe \Vert \weps^i\Vert_{L^2(\Omega_{\varepsilon}^{\lay})} + \sum_{i,j=2}^2 \foe\Vert \partial_i \weps^j \Vert_{L^2(\Omega_{\varepsilon}^M)} + \Vert \weps^3\Vert_{L^2(\Omega_{\varepsilon}^M)} +  \Vert \nabla \weps \Vert_{L^2(\Omega_{\varepsilon}^M)} \le \frac{C}{\varepsilon} \Vert e(\weps)\Vert_{L^2(\Omega_{\varepsilon}^M)}.
\end{align*}
\end{lemma}

With these Korn inequalities at hand, we prove the well-posedness of the microscopic problem together with $\eps$-uniform a priori estimates.
\begin{theorem}\label{theo:existence_result_gen}
Given the assumptions (A1)--(A7), there exists a unique weak solution $\ueps$ in the sense of Definition \ref{def:weak_micro_solution}, which satisfies the following a priori estimate:
\begin{align*}
&\sum_{\pm}\left\{\norm{ \ueps}_{W^{1,\infty}(S;L^2(\OmEps^{\pm}))} 
+ \norm{e(\ueps)}_{L^{\infty}(S;L^2(\OmEps^{\pm}))}
+\|\nabla \ueps\|_{L^{\infty}(S;L^2(\OmEps^{\pm}))} \right\}
\\
&+\frac{1}{\sqrt{\eps}}\norm{ \ueps}_{W^{1,\infty}(S;L^2(\OmEps^\lay))}
+\eps^{\frac{\gamma}{2}}\norm{e(\ueps)}_{L^{\infty}(S;L^2(\OmEps^\lay))}
+\sqrt{\eps} \|\nabla \ueps\|_{L^{\infty}(S;L^2(\OmEps^\lay))} \le C.
\end{align*}
Additionally, for $\gamma = -3$, the following bound for the last two components of $\ueps$ holds:
\begin{align*}
   \frac{1}{\eps^{\frac32}} \sum_{i=2}^3 \|\ueps^i\|_{L^{\infty}(S;L^2(\OmEps^\lay))} \le C.
\end{align*}
\end{theorem}
\begin{proof}
The existence and uniqueness result is standard. It can be proven similar as Theorem 12.4 from \cite{Schw:2018} by using the Galerkin method and relies on a priori estimates for the Galerkin approximation of the solution independent of the Galerkin-approximation parameter. As these a priori estimates are very similar to the $\eps$-independent estimates claimed in the Theorem, we only show these.
%
We multiply equation \eqref{eq:weak_formulation_short} with $\partial_t \ueps$ to obtain for $t_1 \in (0,T)$ (we emphasise that we also have $\partial_t(\varrho_{\eps} \partial_t \ueps) \in L^2(S;\Hgam(\OmEps)^{\ast})$)
 \begin{align*}
 	\int_0^{t_1} \langle \partial_t(\varrho_{\eps} \partial_t \ueps)(t),&\partial_t \ueps(t)\rangle_{\Hgam(\OmEps))^{*},\Hgam(\OmEps))}\,\t +\int_0^{t_1} \int_{\OmEps} A^\eps e(\ueps):e(\partial_t \ueps)\,\x\t \\
 	&= \int_0^{t_1}\langle f^\eps(t),\partial_t \ueps(t)\rangle_{L^2_{\varrho^\eps}(\OmEps)}\,\t+\int_0^{t_1}\langle g(t),\partial_t \ueps(t)\rangle_{L^2(\GamN)}\,\t.
 \end{align*}
Using the properties of $A^{\pm}$ and $A^\lay$, we obtain 
\begin{align*}
 &\frac{1}{2}\varrho_0\left(\norm{\partial_t \ueps(t_1)}^2_{L^2(\OmEps^{+})}+\frac{1}{\eps}\norm{\partial_t \ueps(t_1)}^2_{L^2(\OmEps^\lay)}+\norm{\partial_t \ueps(t_1)}^2_{L^2(\OmEps^{-})}\right)\\
 &+\frac{\alpha}{2}\left(\norm{e(\ueps)(t_1)}^2_{L^2(\OmEps^{+})}+\eps^{\gamma}\norm{e(\ueps)(t_1)}^2_{L^2(\OmEps^\lay)}+\norm{e(\ueps)(t_1)}^2_{L^2(\OmEps^{-})}\right)\\
 &-\frac{1}{2}\varrho_1\left(\norm{u^{+}_{1}}^2_{L^2(\OmEps^{+})}+\frac{1}{\eps}\norm{u^\lay_{1,\eps}}^2_{L^2(\OmEps^\lay)}+\norm{u^{-}_{1}}^2_{L^2(\OmEps^{-})}\right)\\
 &-\frac{\beta}{2}\left(\norm{e(u^{+}_{0})}^2_{L^2(\OmEps^{+})}+\eps^{\gamma}\norm{e(u^\lay_{0,\eps})}^2_{L^2(\OmEps^\lay)}+\norm{e(u^{-}_{0})}^2_{L^2(\OmEps^{-})}\right)\\
 \leq & \int_0^{t_1}\langle f^\eps(t),\partial_t \ueps(t)\rangle_{L^2_{\varrho^\eps}(\OmEps)}\,\t+\int_0^{t_1}\langle g(t),\partial_t \ueps(t)\rangle_{L^2(\GamN)}\,\t.
\end{align*}
Using the assumptions on $f^{\eps}$, we easily obtain that
\begin{align*}
 \int_0^{t_1}\langle f^\eps(t),\partial_t \ueps(t)\rangle_{L^2_{\varrho^\eps}(\OmEps)}\,\t \le C \left( 1  + \int_0^{t_1}\norm{\partial_t \ueps}^2_{L^2(\OmEps^{+})}+\frac{1}{\eps}\norm{\partial_t \ueps}^2_{L^2(\OmEps^\lay)}+\norm{\partial_t \ueps}^2_{L^2(\OmEps^{-})}\t  \right).
\end{align*}
For the term including $g$, we first integrate by parts in time to obtain
\begin{multline*}
\int_0^{t_1}\langle g(t),\partial_t \ueps(t)\rangle_{L^2(\GamN)}\,\t\\ = -\int_0^{t_1} \langle \partial_t g(t) , \ueps(t) \rangle_{L^2(\Gamma_N)} \t +\langle g(t_1),\ueps(t_1)\rangle_{L^2(\Gamma_N)} - \langle g(0), \ueps(0) \rangle_{L^2(\Gamma_N)}.
\end{multline*}
Hence, the regularity of $g$ together with the trace inequality implies   for every $\theta>0$ and a constant $C_{\theta}>0$ (depending on $\theta$ but not on $\eps$)
\begin{align*}
\int_0^{t_1}\langle g(t),\partial_t \ueps(t)\rangle_{L^2(\GamN)}\,\t &\le C_{\theta} + \sum_{\pm} \left\{ \theta \left(\|\ueps(t_1)\|_{H^1(\OmEps^{\pm})}^2 + \|u_0^\pm\|_{H^1(\OmEps^{\pm})}^2 \right)  + C \int_0^{t_1} \|\ueps\|^2_{H^1(\OmEps^{\pm})} \t  \right\}
\\
&\le C_{\theta} + C \int_0^{t_1} \sum_{\pm}  \|e(\ueps)\|_{H^1(\OmEps^{\pm})}^2 + \eps \|e(\ueps)\|_{L^2(\OmEps^\lay)}^2 \t
\\
&\quad + C\theta \bigg( \sum_{\pm} \left\{\|e(\ueps)(t_1)\|_{H^1(\OmEps^{\pm})}^2  + \|e(u_0^{\pm})\|_{H^1(\OmEps^{\pm})}^2 \right\} 
\\
&\qquad\qquad + \varepsilon \|e(\ueps)(t_1)\|_{L^2(\OmEps^\lay)}^2 + \varepsilon \|e(u_{0,\eps}^\lay)\|_{L^2(\OmEps^\lay)}^2\bigg),
\end{align*}
where we have used the Korn inequality from Lemma \ref{prop:estimate_H1_OmEps} in the last step. We emphasise that we do not exclude the case $\Gamma_D \subset \partial \Omega^{\pm}$, and therefore the terms on the right-hand side above associated with the membrane are necessary. 
Combining the previous estimates, using the assumptions on the initial data and choosing $\theta$ small enough, we obtain with an absorption argument
\begin{align*}
&\sum_{\pm} \left\{\norm{\partial_t \ueps(t_1)}^2_{L^2(\OmEps^{\pm})} +\norm{e(\ueps)(t_1)}^2_{L^2(\OmEps^{\pm})} \right\}+\frac{1}{\eps}\norm{\partial_t \ueps(t_1)}^2_{L^2(\OmEps^\lay)}
 +\eps^{\gamma}\norm{e(\ueps)(t_1)}^2_{L^2(\OmEps^\lay)}
 \\
 \le & C \bigg( 1 + \int_0^{t_1} \sum_{\pm}\left\{\norm{\partial_t \ueps}^2_{L^2(\OmEps^{+})} + \|e(\ueps)\|_{H^1(\OmEps^{\pm})}^2  \right\} +\frac{1}{\eps}\norm{\partial_t \ueps}^2_{L^2(\OmEps^\lay)}+ \eps \|e(\ueps)\|_{L^2(\OmEps^\lay)}^2 \t \bigg).
\end{align*}
Using the Gronwall inequality, we get
\begin{align*}
&\norm{\partial_t \ueps}^2_{L^{\infty}(S;L^2(\OmEps^{+}))}+\frac{1}{\eps}\norm{\partial_t \ueps}^2_{L^{\infty}(S;L^2(\OmEps^\lay))}+\norm{\partial_t \ueps}^2_{L^{\infty}(S;L^2(\OmEps^{-}))}\\
 &+\norm{e(\ueps)}^2_{L^{\infty}(S;L^2(\OmEps^{+}))}+\eps^{\gamma}\norm{e(\ueps)}^2_{L^{\infty}(S;L^2(\OmEps^\lay))}+\norm{e(\ueps)}^2_{L^{\infty}(S;L^2(\OmEps^{-}))} \le C.
\end{align*}
Now, the estimates for the gradients  are obtained again by using  the Korn inequality from Lemma \ref{prop:estimate_H1_OmEps}.
By the fundamental theorem of calculus, taking into account the assumptions on the initial data, we immediately obtain 
\begin{align*}
\norm{ \ueps}^2_{L^{\infty}(S;L^2(\OmEps^{+}))}+\frac{1}{\eps}\norm{ \ueps}^2_{L^{\infty}(S;L^2(\OmEps^\lay))}+\norm{ \ueps}^2_{L^{\infty}(S;L^2(\OmEps^{-}))} \le C.
\end{align*}
Finally, for the case $\gamma = -3$, we can improve the estimates on the last two components   by using the Korn inequality of Lemma \ref{lem:KornII}  noting Assumption \ref{ass:zero_Dirichlet_BC_gamma-1-3}  ($\ueps$ has vanishing trace on the lateral boundary in case $\gamma = -3$). Thus, we can apply the Korn inequality from Lemma \ref{lem:KornII} to obtain the desired result.

\end{proof}

\begin{remark}\
 In the case $\gamma = 1$, we control the gradient of $\ueps$ in the layer by the norms from the bulk regions. This seems to be not possible for the estimates in the case $\gamma  =-3$. In this case, we control the last two components of the displacement $\ueps$ and the gradient via the symmetrised gradient in the layer. Therefore, the additional zero boundary condition on $\left(-\frac{\eps}{2},\frac{\eps}{2}\right) \times \partial \omega $  appears to be necessary. 

\end{remark}

\section{Compactness results}
\label{sec:compactness}

The a priori estimates of the previous section allow to use compactness results yielding limits of the sequences of solutions to the microscopic problem of Definition \ref{def:weak_micro_solution}. The precise results are collected in this section before the problems satisfied by the limits are identified in the next section.

We formulate the compactness results using the unfolding operator, cf.~Definition \ref{def:unfolding_operator}, which satisfies the following properties (cf.~\cite[Proposition 3.1]{Griso:2015}).
\begin{proposition}\label{prop:properties_unfolding_operator}\
    \begin{enumerate}[label = (\roman*)]
        \item For all $\phi\in L^1(\Omega^\lay_{\eps})$: \[ \int_{\Omega^\lay_{\eps}}\phi\,\x=\eps\int_{\omega\times Y_0}\TEps(\phi)(x',y)\,\x'\y. \]
        \item For all $\phi\in L^2(\Omega^\lay_{\eps})$: \[ \norm{\phi}_{L^2(\Omega^\lay_{\eps})} = \sqrt{\eps}\norm{\TEps(\phi)}_{L^2(\omega\times Y_0)}. \]
        \item Let $\phi\in H^1(\Omega^\lay_{\eps})$. Then, $\nabla_y(\TEps(\phi))=\eps\TEps(\nabla\phi)$ a.e. in $\omega\times Y_0$.
    \end{enumerate}
\end{proposition}
We define
\begin{equation}\label{def:H1_tilde_per_Y0}
    \widetilde{H}^1_{\#}(Y_0)\coloneqq\{\phi\in H^1(Y_0)\,|\, \phi\text{ is } Y'\text{-periodic with respect to }y'=(y_2,y_3)\}.
\end{equation} 
With this at hand, we can state the compactness results for the microscopic sequence $\ueps$. The results are more or less direct consequences of the a priori estimates from Theorem \ref{theo:existence_result_gen}. 
We formulate the results separately in the bulk domains and the thin layer, starting with the convergence in the bulk regions which does not require any unfolding:
\begin{proposition}\label{prop:compactness_bulk_domains}
There exist limit functions
\begin{align*}
u^{\pm} \in L^{\infty}(S;H^1_{\GamD}(\Omega^{\pm}))^3 \cap W^{1,\infty}(S;L^2(\Omega^{\pm}))^3,
\end{align*}
such that, up to a subsequence, 
\begin{align*}
\chi_{\OmEps^\pm}\ueps &\rightarrow u^\pm &\mbox{  in }& L^2(S\times\Omega^\pm)^3,
\\
\chi_{\OmEps^\pm}\nabla \ueps &\weak \nabla u^\pm  &\mbox{ weakly in }& L^2(S\times\Omega^\pm)^3,
\\
\chi_{\OmEps^\pm}\partial_t \ueps &\rightarrow \partial_t u^\pm  &\mbox{   in }& L^2(S\times\Omega^\pm)^3,
\\
\ueps|_{S_{\eps}^{\pm}} &\weak u^{\pm}|_{\omega}  &\mbox{ weakly in }& L^2(S \times \omega)^3.
\end{align*}
Furthermore, the initial condition $u^{\pm} (0) = u_0^{\pm}$ holds.
\end{proposition}
\begin{proof}
The weak convergences follow directly from the a priori estimates in Theorem \ref{theo:existence_result_gen} and standard weak compactness in $L^2$. The  strong convergence of $\chi_{\OmEps^{\pm}} \ueps$ in $L^2(S \times \Omega^{\pm})^3$  and, therefore, of $\ueps|_{S_{\eps}^{\pm}}$ in $L^2(S \times \omega)^3$ in particular are obtained by similar arguments as in \cite[Section 5.1]{Neuss:2007}, see also the proof of \cite[Proposition 3]{gahn2025effective} for a slightly different argument. 
\end{proof}
\begin{remark}\label{rem:strong_convergence_bulk}
Until now, we have no information about $\partial_{tt} u^{\pm}$ and the initial condition for $\partial_t u^{\pm}(0)$. We will obtain results for these quantities from the macroscopic model later. However, we emphasise that it is possible to show a priori estimates for $\partial_{tt} \ueps $ in $L^2(S,H_{\Gamma_D}^1(\Omega)')^3$, obtaining more regularity, and thus also (weak) compactness for the second derivative.

\end{remark}
Next, we show the compactness result in the thin layer, where we distinguish between the three cases $\gamma \in \{1,-1,-3\}$. For this, we consider the unfolded sequence $\TEps(\ueps)$.
\begin{proposition}[The case $\gamma = 1$]\label{prop:compactness_layer_gamma_1}
There exists 
\begin{align*}
u^\lay \in  L^{\infty}(S;L^2(\omega,\widetilde{H}_{\#}^1(Y_0)))^3 \cap W^{1,\infty}(S;L^2(\omega\times Y_0))^3
\end{align*}
such that, up to a subsequence,
\begin{align*}
\TEps(\ueps) &\weak u^\lay &\mbox{ weakly in }& L^2(S \times \omega;H^1(Y_0))^3,
\\
\TEps (\partial_t \ueps) &\weak \partial_t u^\lay &\mbox{ weakly in }& L^2(S \times \omega \times Y_0)^3.
\end{align*}
In addition, the initial condition $u^\lay(0) = u_0^\lay$ holds.
\end{proposition}
\begin{proof}
This follows from the a priori estimates in Theorem \ref{theo:existence_result_gen} and the properties of the unfolding operator in Proposition \ref{prop:properties_unfolding_operator} together with standard (weak) compactness in $L^2$.
\end{proof}
Again, we have no information about $\partial_{tt} u^\lay$, which will be obtained from the macromodel later.
\begin{proposition}[The case $\gamma = -1$]\label{prop:compactness_layer_gamma_-1} For $\gamma = -1$ (and noting the additional assumption \ref{ass:zero_Dirichlet_BC_gamma-1-3} in this case), there exist 
\begin{align*}
u^\lay &\in L^2(S;H_0^1(\omega)^2 \times L^2(\omega))\cap W^{1,\infty}(S;L^2(\omega))^3,
\\
u^{\lay,1} &\in L^2(S \times \omega,\widetilde{H}_{\#}(Y_0)/\R)^3
\end{align*}
such that, up to a subsequence,
\begin{align*}
\TEps(\ueps) &\weak u^\lay &\mbox{ weakly in }& L^2(S\times \omega;H^1(Y_0))^3,
\\
\frac{1}{\eps} e_y(\TEps(\ueps)) = \TEps(e(\ueps)) &\weak e_{x'}(\hat{u}^\lay) + e_y(u^{\lay,1}) &\mbox{ weakly in }& L^2(S\times \omega \times Y_0)^{3\times 3},
\\
\TEps (\partial_t \ueps) &\weak \partial_t u^\lay &\mbox{ weakly in }& L^2(S \times \omega \times Y_0)^3,
\end{align*}
with $\hat{u}^\lay := (u_1^\lay,u_2^\lay)$. 
Furthermore, $u^\lay(0) = u_0^\lay$ holds.
\end{proposition}
\begin{proof}
This result follows from \cite[Lemma 3]{gahn2025b_effective} (for which we need the vanishing Dirichlet boundary condition on $(- \frac{\eps}{2} , \frac{\eps}{2}) \times\partial \omega $ because the proof of this lemma also uses the Korn  inequality from Lemma \ref{lem:KornII}) and our a priori estimates). We emphasise that the weak convergence of the unfolded sequence $\TEps (\ueps)$ is equivalent to the (weak) two-scale convergence of $\ueps$ in the thin layer. The convergence of the gradient $\nabla_y \TEps(\ueps)$ follows since this sequence is bounded in $L^2$. 
\end{proof}
Finally, we give the result for $\gamma  =-3$, where we obtain a Kirchhoff--Love displacement as the limit function.
\begin{proposition}[The case $\gamma = -3$]\label{prop:compactness_layer_gamma_-3} For $\gamma = -3$ (and noting the additional assumption \ref{ass:zero_Dirichlet_BC_gamma-1-3} in this case), there exist
\begin{align*}
u^\lay &\in L^2(S;H^2_0(\omega))^3 \cap W^{1,\infty}(S;L^2(\omega))^3,
\\
\hat{u}^\lay &\in L^2(S;H_0^1(\omega))^2, 
\\
u^{\lay,2} &\in L^2(S\times \omega,\widetilde{H}_{\#}^1(Y_0)/\R)^3
\end{align*}
with $[u^\lay]_2 = [u^\lay]_3 = 0$, such that, up to a subsequence, 
\begin{align*}
\TEps(\ueps)_1 &\weak [u^\lay]_1 &\mbox{ weakly in }& L^2(S\times \omega \times Y_0),
\\
\frac{1}{\eps}[\TEps(\ueps)]_{\alpha} &\weak [\hat{u}^\lay]_{\alpha} &\mbox{ weakly in }& L^2(S\times \omega \times Y_0),
\\
\frac{1}{\eps^2} e_y(\TEps(\ueps)) = \frac{1}{\eps} \TEps(e(\ueps)) &\weak e_{x'}(\hat{u}^\lay) - y_3 \nabla_{x'}^2 [u^\lay]_1 + e_y(u^{\lay,2}) &\mbox{ weakly in }& L^2(S \times \omega \times Y_0)^{3\times 3},
\\
\TEps(\partial_t\ueps)_1 &\weak \partial_t [u^\lay_0]_1  &\mbox{ weakly in }& L^2(S\times \omega \times Y_0),
\end{align*}
where $\alpha = 2,3$ and we understand $\hat{u}^\lay$ as a function mapping into $\R^3$ by setting the first component equal to zero. 
Furthermore, the initial condition $[u^\lay]_1(0) = [u_0^\lay]_1$ holds.
\end{proposition}
\begin{proof}
This result follows from our a priori estimates using \cite[Theorem 3]{gahn2021two}.
\end{proof}

In the next Proposition, we formulate the interface conditions between the limits $u^{\pm}$ in the bulk regions and the limit functions $u^\lay$ from Proposition \ref{prop:compactness_layer_gamma_1} --\ref{prop:compactness_layer_gamma_-3} for the different choices of $\gamma$. 

\begin{proposition}\label{prop:interface_condition}
For the limit functions $u^{\pm}$ from Proposition \ref{prop:compactness_bulk_domains} and $u^\lay$ from Proposition \ref{prop:compactness_layer_gamma_1} -- \ref{prop:compactness_layer_gamma_-3}, respectively,  
\begin{align*}
    u^{\pm}|_{\omega} = u^\lay|_{S_Y^{\pm}} \qquad \mbox{ on } S \times \omega \times S_Y^{\pm}
\end{align*}
holds. In particular, for $\gamma = 1$ the trace of  $u^\lay$ is constant on $S_{Y}^{\pm}$ with respect to $y$; for $\gamma \in \{-1,-3\}$, (as $u^\lay$ is independent of $y$) we otbain the continuity of the displacements $u^+$ and $u^-$ across $\omega$:
\begin{align*}
    u^+|_{\omega} = u^-|_{\omega} \qquad\mbox{on } S \times \omega \quad\mbox{for } \gamma \in \{-1,-3\}.
\end{align*}
\end{proposition}
\begin{proof}
Similar results can be found for example in \cite{gahn2022derivation,gahn2025effective,Neuss:2007}. However, for the sake of completeness, we briefly sketch the idea of the proof which is the same for all choices of $\gamma$. From the a priori estimates in Theorem \ref{theo:existence_result_gen}, we have 
\begin{align*}
    \frac{1}{\sqrt{\eps}} \|\ueps\|_{L^2(S \times \OmEps^\lay)} + \sqrt{\eps} \|\nabla \ueps\|_{L^2(S\times \OmEps^\lay)} \le C.
\end{align*}
This implies (as already obtained in Proposition \ref{prop:compactness_layer_gamma_1} in the case $\gamma =1$) that 
\begin{align*}
\TEps(\ueps) \weak u^\lay \qquad\mbox{ weakly in } L^2(S\times \omega , H^1(Y_0))^3.
\end{align*}
In particular, we get $\TEps(\ueps)|_{S_Y^{\pm}} \weak u^\lay|_{S_Y^{\pm}} $ weakly in $L^2(S \times \omega \times S_{Y}^{\pm})^3.$ Further, the strong convergence of $\ueps|_{S_{\eps}^{\pm}}$ in $L^2(S\times \omega)^3$ implies the convergence of the associated unfolded sequence in $L^2(S \times \omega \times Y')^3$ to the same limit. Now, the desired result follows easily from the interface condition on $S_{\eps}^{\pm}$.
\end{proof}

We emphasise that, for $\gamma = -3$, we have $u^{\pm}_2 = u^{\pm}_3 = 0$ on $\omega.$ Further, in the case $\gamma = -1$, we immediately obtain the improved regularity $u^\lay_3 \in L^2(S;H^{\frac12}(\omega))$ by the trace theorem.

\section{Identification of the limit problems}\label{sec:identification}

In this section, we pass to the limit $\eps \to 0$ in the microscopic problem to obtain the macroscopic limit models stated in Section \ref{sec:mainresults} satisfied by the limit functions obtained in the previous section. The three different cases $\gamma \in \{1,-1,-3\}$ are considered separately. As the compactness results for the sequences in the bulk from Proposition \ref{prop:compactness_bulk_domains} are the same in all cases, only the convergences in the layer are different, which leads to different effective interface conditions across $\omega$.

\subsection{The case $\gamma  = 1$}

We start with the case $\gamma = 1$. This can be considered a critical case, because the two-scale limit $u^\lay$ in the thin layer depends on both the macroscopic variable $x'$ and the microscopic variable $y$. In a first step, we derive the two-scale homogenised model $\eqref{eq:macro_model_gamma1_strong}$. Since the limit displacement $u^\lay$ depends additionally on the microscopic variable, the numerical approximation of the solution to this problem is expensive (since a PDE on $Y_0$ has to be solved in every macroscopic point $x' \in \omega$). Therefore, in a second step, we replace the equation for $u^\lay$ by effective interface conditions. For this, we extract cell problems from the two-scale homogenised problem, which can be solved independently, to write this problem as a macroscopic equation dependent on $x$ only, cf.~\eqref{eq:macro_model_gamma1_effective}.

In the following theorem, we derive the two-scale homogenised problem which already includes all the information of the macroscopic model. For convenience, we first recall the definition of the solution space $\spaceH_1$, see \eqref{def:spaceH_1},
\begin{multline*}
\spaceH_1= \bigg\{(v^+,v^\lay,v^-) \in H^1_{\GamD\cap\partial\Omega^+}(\Omega^+)^3  \times  L^2(\omega,\widetilde{H}^1_{\#}(Y_0))^3 \times H^1_{\GamD\cap\partial\Omega^-}(\Omega^-)^3  \, | \\ \,
 v^{\pm} = v^\lay \mbox{ on } \omega \times S_Y^{\pm}  \bigg\}.
\end{multline*}

\begin{theorem}\label{thm:main_convergence_two_scale_model}
Let $\gamma  = 1$ and  $\ueps $ be the weak solution of the microscopic equation \eqref{eq:weak_formulation}. Then, the triple of limit functions $(u^+,u^\lay,u^-) \in L^2(S,\spaceH_1)$ from Propositions \ref{prop:compactness_bulk_domains} and \ref{prop:compactness_layer_gamma_1} is the unique solution to the macroscopic problem $\eqref{eq:macro_model_gamma1_strong}$.
\end{theorem}

\begin{proof}
First of all, we note that from Proposition \ref{prop:interface_condition} we get that $(u^+,u^\lay,u^-) \in L^2(S,\spaceH_1).$ Now, we choose the following structure of the test function in the weak equation $\eqref{eq:weak_formulation}$:  
\begin{align*}
 v^\eps(t,x)=\begin{cases}
        v^\pm(t,x), &(t,x)\in S\times\OmEps^\pm,\\
        v^\lay \left(t,x',\frac{x}{\eps}\right), &(t,x)\in S\times\OmEps^\lay,
    \end{cases}
\end{align*}
where $v^\pm\in L^2(S,H^1_{\GamD\cap\partial\Omega^\pm}(\Omega^\pm))$ with $\partial_t v^\pm\in L^2(S\times\Omega^\pm)$, $v^\pm(T)=0$ and $v^\lay\in L^2(S,C^\infty_0(\omega,\widetilde{H}^1_{\#}(Y_0)))$ with $\partial_tv^\lay\in L^2(S,C^\infty_0(\omega,L^2(Y_0)))$,  $v^\lay(T)=0$ such that $v^\pm|_{\omega}=v^\lay|_{\omega \times S_Y^{\pm}}$ on $ S\times\omega\times Y'$.  From the unfolding method, see for example \cite{Cior:2018}, it is well-known (and easy to check) that, for a function $\phi^\lay_{\eps}(t,x):= \phi^\lay\left(x',\frac{x}{\eps}\right)$ with $\phi^\lay \in C^0(\overline{\omega},\widetilde{H}^1_{\#}(Y_0))$,  $\TEps(\phi^\lay_{\eps} ) \rightarrow \phi^\lay$ in $L^2(\omega \times Y_0)$ holds. In particular, we get (as the time variable only acts as a parameter)
\begin{align*}
\TEps(v^\eps)&\rightarrow v^\lay &\mbox{ strongly in }& L^2(S\times\omega;H^1(Y_0))^3,
\\
e_y(\TEps(v^{\eps})) = \eps \TEps(e(v^{\eps})) &\rightarrow e_y(v^\lay) &\mbox{ strongly in }& L^2(S\times \omega \times Y_0)^{3\times 3},
\end{align*}
and similarly for the time derivative,
\begin{align*}
\TEps(\partial_t v^\eps)&\rightarrow\partial_tv^\lay &\mbox{ strongly in }& L^2(S\times\omega\times Y_0)^3.
\end{align*}
We rewrite the weak formulation \eqref{eq:weak_formulation} using the test functions from above and the periodic unfolding operator to obtain
\begin{align*}
	-&\sum_{\pm} \left\{\int_0^T\int_{\Omega^\pm} \varrho^\pm\chi_{\OmEps^\pm}\partial_t \ueps\cdot\partial_t v^\pm\,\x\t - \int_0^T\int_{\Omega^\pm} A^\pm\chi_{\OmEps^\pm}e(\ueps):e(v^\pm)\,\x\t \right\}\\
    &-\int_0^T\int_{\omega\times Y_0} \TEps(\varrho^\lay_\eps)\TEps(\partial_t \ueps)\cdot\TEps(\partial_t v^\eps)\,\x'\y\t\\
    &+\eps^2\int_0^T\int_{\omega\times Y_0} \TEps(A^\lay_\eps)\TEps(e(\ueps)):\TEps(e(v^\eps))\,\x'\y\t\\
	=& \sum_{\pm}\bigg\{\int_0^T\int_{\Omega^\pm} \varrho^\pm\chi_{\OmEps^\pm} f^\pm\cdot v^\pm\,\x\t +\int_0^T\int_{\GamN\cap\partial\Omega^\pm} g\cdot v^\pm\,\mathrm{d}\sigma(x)\t \\
    &+\int_{\Omega^\pm}\varrho^\pm\chi_{\OmEps^\pm}u^\pm_1\cdot v^\pm(0,x)\,\x \bigg\}+\int_0^T\int_{\omega\times Y_0} \TEps(\varrho^\lay_\eps)\TEps(f^\lay_\eps)\cdot\TEps(v^\eps)\,\x'\y\t \\
    &+\int_{\omega\times Y_0} \TEps(\varrho^\lay_\eps)\TEps(u^\lay_{1,\eps})\cdot\TEps(v^\eps(0,x))\,\x'\y.
\end{align*}
Passing to the limit $\eps\to 0$, we obtain \eqref{eq:macro_model_two_scale_gamma_1}.

It remains to establish uniqueness of the solution. For this purpose, assume there exist two weak solutions $u_\mathrm{a}:=(u^+_\mathrm{a},u^\lay_\mathrm{a},u^-_\mathrm{a})$ and $u_\mathrm{b}:=(u^+_\mathrm{b},u^\lay_\mathrm{b},u^-_\mathrm{b})$ and define $(u^+,u^\lay,u^-)\coloneqq (u^+_\mathrm{b}-u^+_\mathrm{a},u^\lay_\mathrm{b}-u^\lay_\mathrm{a},u^-_\mathrm{b}-u^-_\mathrm{a})$. Due to linearity, $u^\pm(0,x)=0$ in $\Omega^\pm$, $u^\lay(0,x',y)=0$ in $\omega\times Y_0$, $u^\pm(t,0,x')=u^\lay(t,x',\pm\frac{1}{2},y')$ for a.e.~$(t,x',y')\in S\times\omega\times Y'$ and
\begin{equation}\label{eq:uniqueness}
\begin{split}
	-&\sum_{\pm}\int_0^T\int_{\Omega^\pm} \varrho^\pm\partial_t u^\pm\cdot \partial_t v^\pm\,\x\t - \int_0^T\int_{\Omega^\pm} A^\pm(x)e(u^\pm):e(v^\pm)\,\x\t \\
    -&\int_0^T\int_{\omega\times Y_0} \varrho^\lay\partial_t u^\lay\cdot \partial_t v^\lay\,\x'\y\t +\int_0^T\int_{\omega\times Y_0} A^\lay(x)e_y(u^\lay):e_y(v^\lay)\,\x'\y\t=0
\end{split}
\end{equation}
for all $v:=(v^+,v^\lay,v^-) \in L^2(S,\spaceH_1)$ with $v(T) = 0$,  $\partial_t v^\pm\in L^2(S\times\Omega^\pm)^3$ and $\partial_tv^\lay\in L^2(S\times\omega\times Y_0)^3$. Now, let $0\leq s\leq T$ and 
\begin{align*}
  v^\pm(t,x) \coloneqq-\int_t^T\chi_s(\tau)u^\pm(\tau,x)\,\mathrm{d}\tau, \quad  v^\lay(t,x',y) \coloneqq-\int_t^T\chi_s(\tau)u^\lay(\tau,x',y)\,\mathrm{d}\tau, 
\end{align*} 
where $\chi_s$  is the characteristic function of the interval $[0,s]$.  
Then, $v$ is an admissible test function, $v^\pm(t,\cdot)=v^\pm(T,\cdot)=v^\lay(t,\cdot,\cdot)=v^\lay(T,\cdot,\cdot)=0$ for $s\leq t\leq T$, 
$v$ is absolutely continuous in $[0,T]$, $\partial_t v^\pm(t,x)=\chi_s(t)u^\pm(t,x)$ a.e.~in $S\times\Omega$ and $\partial_t v^\lay(t,x',y)=\chi_s(t)u^\lay(t,x',y)$ a.e.~in $S\times\omega\times Y_0$. Using this test function in \eqref{eq:uniqueness}, we obtain by standard arguments
\begin{align*}
	0=&-\sum_{\pm}\int_0^s\int_{\Omega^\pm}\varrho^\pm\partial_t u^\pm\cdot u^\pm\,\x\t+
    \frac{1}{2}\int_{\Omega^\pm} A^\pm e(v^\pm)(0,x):e(v^\pm)(0,x)\,\x\\
    &-\int_0^s\int_{\omega\times Y_0}\varrho^\lay\partial_t u^\lay\cdot u^\lay\,\x'\y\t-\frac{1}{2}\int_{\omega\times Y_0} A^\lay e_y(v^\lay)(0,x',y):e_y(v^\lay)(0,x',y)\,\x'\y,
\end{align*}
where we have used the symmetry of $A^\pm,A^\lay$. The coercivity of $A^\pm,A^\lay$ yields
\begin{align*}
	0\geq & \sum_{\pm}\frac{1}{2}\varrho_0\norm{u^\pm(s,\cdot)}^2_{[L^2(\Omega^\pm)]^3}+\frac{\alpha}{2}\norm{v^\pm(0,\cdot)}^2_{H^1(\Omega^\pm)}+\frac{1}{2}\varrho_0\norm{u^\lay(s,\cdot,\cdot)}^2_{[L^2(\omega\times Y_0)]^3}\\
    &+\frac{\alpha}{2}\norm{v^\lay(0,\cdot,\cdot)}^2_{L^2(\omega,H^1(Y_0))}
\end{align*}
for a.e.~$s\in S$. Thus, $u_\mathrm{a}=u_\mathrm{b}$ a.e.
\end{proof}

From the weak equation $\eqref{eq:macro_model_two_scale_gamma_1}$, we immediately obtain a representation of the jumps of the displacement and the stresses across $\omega$. For this, we first introduce an auxiliary problem: 
Letting
\begin{align*}
V\coloneqq\{\varphi\in H^1(Y_0)^3\,|\, \varphi\text{ periodic in }Y', \varphi\equiv \text{const.~on }S^+_Y\cup S^-_Y\},
\end{align*}
find $\eta^{(k)}\in V$, $k\in\{1,2,3\}$, with $\int_{Y_0}\eta^{(k)}(y)\,\y =0$ such that
\begin{align}\label{cell_problem_eta^k_jump_weak}
\int_{Y_0} A^\lay e(\eta^{(k)}):e(\varphi)\,\y = \int_{S^+_Y}\varphi_k\,\Sy-\int_{S^-_Y}\varphi_k\,\Sy
\end{align}
for all $\varphi\in V$. 
We immediately see that this problem has a unique solution for each $k$ up to rigid-body motions applying Korn's inequality from \cite[Theorem 2.5]{Olei:1992} and the Lax--Milgram Theorem. Noting that functions in $V$ are constant on $S^+_Y\cup S^-_Y$, only rigid-body motions which are (arbitrary) translations are actually possible. Further noting that $\eta^{(k)}$ has vanishing average, this translation is uniquely defined. Thus, the auxiliary problem has a unique solution without any further constraints, the strong formulation of which is given by
\begin{align}
\begin{aligned}\label{cell_problem_eta^k_jump}
		-\nabla\cdot (A^\lay e(\eta^{(k)})) &= 0 &\text{in } Y_0,\\
        -(A^\lay e(\eta^{(k)}))n&=0&\text{on }\partial Y_1,\\
        \int_{S^\pm_Y} A^\lay e(\eta^{(k)})n\,\Sy &= \pm e_k,\\
        \frac{1}{\abs{Y_0}}\int_{Y_0}\eta^{(k)}(y)\,\y &=0, \\
        \eta^{(k)} \mbox{ constant on } S_Y^{\pm,}
\end{aligned}
\end{align}
where $e_k$ is the $k$-th unit vector. 
\begin{proposition}\label{prop:jump_conditions_gamma_1}
The transmission conditions on the interface $\omega$ can be expressed as
\begin{align*}
 u^+_k(t,0,x')-u^-_k(t,0,x') = \int_{Y_0} A^\lay e_y(\eta^{(k)}) : e_y(u^\lay) \,\y 
\end{align*}  
    for $k\in\{1,2,3\}$ and, in the distributional sense,
\begin{align*}
 (A^{+}e(u^{+})n - A^{-}e(u^{-})n)(t,0,x') = \int_{Y_0} (\varrho^\lay f^\lay - \partial_t(\varrho^\lay\partial_t u^\lay))\,\y,
\end{align*}
    where $n$ is the normal from $\Omega^+$ to $\Omega^-$, i.e.~$n=(-1,0,0)$. 
\end{proposition}
\begin{proof}
The first equation is obtained directly by testing the equation for $\eta^{(k)}$ with $u^\lay$. For the second identity, we choose $\phi \in H_0^1(\Omega)^3$ and put $\phi^{\pm} := \phi|_{\Omega^{\pm}}$ and $\phi^\lay := \phi|_{\omega}$. Now, choosing $(\phi^+,\phi^\lay,\phi^-) \in \spaceH_1$ as a test function in equation $\eqref{eq:two_scale_macro_gamma1_aux}$ and integrating (formally) by parts gives the result.
\end{proof}

\subsubsection{Effective interface conditions}
\label{sec:effective_interface_conditions}
Our aim now is to derive a macroscopic model in terms of the bulk displacements $(u^+,u^-)$ only. Until now, the jump conditions across $\omega$ for the displacement and the normal stress depend on the function $u^\lay$ (as well as the two-scale homogenised model in Theorem \ref{thm:main_convergence_two_scale_model}). In what follows, we express the function $u^\lay$ via the traces of $u^{\pm}$ on $\omega$ and suitable cell problems. This allows to formulate a macromodel for $(u^+,u^-)$ with effective interface conditions across $\omega$ independent of $u^\lay$.

For this result, we make the following additional assumptions:
\begin{enumerate}
[label = (A\arabic*$^{\prime}$)]
\item\label{ass:additional_solution_trace} Improved regularity of the bulk solution: $\partial_t u^{\pm}|_{\omega} \in H^1(S;L^2(\omega))^3.$
\item\label{ass:additional_initial_value} Compatibility of the initial value $u_0^\lay$: 
\begin{align*}
-\nabla_y \cdot(A^\lay e_y(u_0^\lay)) &= 0 &\mbox{ in }& \omega \times Y_0,
\\
- A^\lay e_y(u_0^\lay)\nu &= 0 &\mbox{ on }& \omega \times \partial Y_1,
\\
u_0^\lay|_{S_{Y}^{\pm}} &= u_0^\pm|_{\omega} &\mbox{ on }& \omega \times S_Y^{\pm},
\\
u_0^\lay \mbox{ is }& Y'\mbox{-periodic}.
\end{align*}
\end{enumerate}

\begin{remark}
While condition \ref{ass:additional_initial_value} is an assumption on the data, the condition \ref{ass:additional_solution_trace} is an assumption on the solution itself, which is, of course, not known a priori. However, (under reasonable assumptions on the data) we expect it is possible to obtain this regularity from the two-scale homogenised model $\eqref{eq:macro_model_two_scale_gamma_1}$. 
\end{remark}

In order to proceed, we write $u^\lay$ in the form
\begin{align}\label{eq:decomposition_u^M}
    u^\lay = u_0^\lay + \tilde{u}_1^\lay + u_f^\lay + w^\lay,
\end{align}
where the functions $\tilde{u}_1^\lay$, $u_f^\lay$ and $w^\lay$ solve the following cell problems (we refer to Appendix \ref{sec:appendix_wave_equation} for more details on weak solutions to such cell problems):
\begin{itemize}
\item $\tilde{u}_1^\lay$ is the unique weak solution to  
\begin{align}
\begin{aligned}\label{eq:cell_problem_tilde_u1^M}
\partial_t (\varrho^\lay \partial_t \tilde{u}_1^\lay) - \nabla_y \cdot(A^\lay e_y(\tilde{u}_1^\lay)) &= 0 &\mbox{ in }& S \times \omega \times Y_0,
\\
- A^\lay e_y(\tilde{u}_1^\lay)\nu &= 0 &\mbox{ on }& S \times \omega \times \partial Y_1,
\\
\tilde{u}_1^\lay|_{S_{Y}^{\pm}} &= 0 &\mbox{ on }& S \times \omega \times S_Y^{\pm},
\\
\tilde{u}_1^\lay(0) &= 0 &\mbox{ in }& \omega \times Y_0,
\\
\partial_t \tilde{u}_1^\lay (0) &= u_1^\lay &\mbox{ in }& \omega \times Y_0,
\\
\tilde{u}_1^\lay \mbox{ is } Y'\mbox{-periodic}.
\end{aligned}
\end{align}

\item $u_f^\lay$ is the unique weak solution to 
\begin{align*}
\partial_t (\varrho^\lay \partial_t u_f^\lay) - \nabla_y \cdot(A^\lay e_y(u_f^\lay)) &= \varrho^\lay f^\lay &\mbox{ in }& S \times \omega \times Y_0,
\\
- A^\lay e_y(u_f^\lay)\nu &= 0 &\mbox{ on }& S \times \omega \times \partial Y_1,
\\
u_f^\lay|_{S_{Y}^{\pm}} &=0 &\mbox{ on }& S \times \omega \times S_Y^{\pm},
\\
u_f^\lay(0) &= 0 &\mbox{ in }& \omega \times Y_0,
\\
\partial_t u_f^\lay (0) &= 0 &\mbox{ in }& \omega \times Y_0,
\\
u_f^\lay \mbox{ is } Y'\mbox{-periodic}.
\end{align*}

\item $w^\lay$ is the unique weak solution to 
\begin{align*}
\partial_t (\varrho^\lay \partial_t w^\lay) - \nabla_y \cdot(A^\lay e_y(w^\lay)) &= 0&\mbox{ in }& S \times \omega \times Y_0,
\\
- A^\lay e_y(w^\lay)\nu &= 0 &\mbox{ on }& S \times \omega \times \partial Y_1,
\\
w^\lay|_{S_{Y}^{\pm}} &= u^{\pm}|_{\omega} - u_0^{\pm}|_{\omega} &\mbox{ on }& S \times \omega \times S_Y^{\pm},
\\
w^\lay(0) &= 0 &\mbox{ in }& \omega \times Y_0,
\\
\partial_t w^\lay (0) &= 0 &\mbox{ in }& \omega \times Y_0,
\\
w^\lay \mbox{ is }& Y'\mbox{-periodic}.
\end{align*}
\end{itemize}
The idea behind this decomposition is that each term on the right-hand side of $\eqref{eq:decomposition_u^M}$ is associated with an initial condition ($u_0^\lay$ and $\tilde{u}_1^\lay$), a force term ($u_f^\lay$) or the coupling to the bulk solutions via the boundary conditions ($w^\lay$). The most interesting and critical term seems to be the function $w^\lay$ which describes the coupling to the bulk domains. Owing to the linear character of the equation for $w^\lay$, we can express $w^\lay$ via suitable cell problems (not depending on the bulk solutions) and coefficients depending on $u_0^{\pm}$. More precisely, we have, using the uniqueness of the problem,
\begin{align*}
w^\lay(t,x',y):=& \sum_{\pm}\sum_{i=1}^3 \bigg\{\int_0^t \partial_t u_i^{\pm} (s,0,x') \chi_i^{\pm}(t-s,x',y) \s + \int_0^t \partial_{tt} u_i^{\pm}(s,0,x') \eta_i^{\pm}(t-s,x',y)\s\\
&+ \partial_t u_i^{\pm}(0,0,x') \eta_i^{\pm}(t,x',y)\bigg\},
\end{align*}
where $\chi_i^{\pm}$ and $\eta_i^{\pm}$  solve the following cell problems for a function $\phi_i^{\pm} \in H^1(Y_0)^3$ which is $Y'$-periodic and fulfils $\phi_i^{\pm} = e_i$ on $S_Y^{\pm}$ and $\phi_i^{\pm} = 0$ on $S_Y^{\mp}$: 
\begin{align}
\begin{aligned}\label{eq:cell_problem_chi^pm}
\partial_{t}(\varrho^\lay\partial_t\chi_i^{\pm}) - \nabla_y \cdot (A^\lay e_y(\chi_i^{\pm})) &= 0 &\mbox{ in }& S\times \omega \times Y_0,
\\
\chi_i^{\pm} &= \phi_i^{\pm} &\mbox{ on }& S  \times\omega\times (S^+_Y \cup S^-_Y),
\\
-A^\lay e_y(\chi_i^{\pm})\nu &=0 &\mbox{ on }& S \times\omega\times\partial Y_1,
\\
\chi_i^{\pm}(0) &= \phi_i^{\pm} &\mbox{ in }& \omega\times Y_0,
\\
\partial_t \chi_i^{\pm} (0) &= 0 &\mbox{ in }& \omega\times Y_0,
\\
\chi_i^{\pm}\text{ periodic with respect to }&Y'
\end{aligned}
\end{align}
and
\begin{align}
\begin{aligned}\label{eq:cell_problem_eta^pm}
\partial_{t}(\varrho^\lay \partial_t\eta_i^{\pm}) - \nabla_y \cdot (A^\lay e_y(\eta_i^{\pm})) &= 0 &\mbox{ in }& S\times \omega \times Y_0,
\\
\eta_i^{\pm} &=0 &\mbox{ on }& S\times\omega\times (S^+_Y \cup S^-_Y),
\\
-A^\lay e_y(\eta_i^{\pm})\nu &=0 &\mbox{ on }& S \times\omega\times\partial Y_1,
\\
\eta_i^{\pm}(0) &= 0 &\mbox{ in }& \omega\times Y_0,
\\
\partial_t \eta_i^{\pm} (0) &= -\phi_i^{\pm} &\mbox{ in }& \omega\times Y_0,
\\
\eta_i^{\pm}\text{ periodic with respect to }&Y'.
\end{aligned}
\end{align}
We note that
\begin{align*}
\partial_t w^\lay(t,x',y)=& \sum_{\pm} \sum_{i=1}^3 \bigg\{ \partial_t u_i^{\pm} (t,0,x')\phi^\pm_i(y) + \partial_t u_i^{\pm}(0,0,x') \partial_t\eta_i^{\pm}(t,x',y)\\
    &+ \int_0^t \partial_t u_i^{\pm}(s,0,x') \partial_t \chi_i^{\pm}(t-s,x',y) + \partial_{tt} u_i^{\pm}(s,0,x') \partial_t\eta_i^{\pm}(t-s,x',y)\s \bigg\}.
\end{align*}
Considering the regularity stated for a weak solution in Appendix \ref{sec:appendix_wave_equation}, this equation is not fulfilled pointwise almost everywhere, but in the $L^2$-sense with respect to time. (Nonetheless, solutions to the cell problems above are continuous with respect to time and therefore also the pointwise equation makes sense.)  
It is easy to check that $w^\lay|_{S_Y^{\pm}} = u^{\pm}|_{\omega} - u_0^{\pm}|_{\omega}$ on $S\times \omega \times S_Y^{\pm}$ and also to verify the initial conditions $w^\lay(0) = \partial_t w^\lay (0) = 0$. It remains to establish the weak equation for $w^\lay$, which follows by a lengthy but elementary calculation using Lemma \ref{lem:aux_weak_sol_wave_equation} and we skip the details.

Finally, we consider the term $u_f^\lay$ including the force term $f^\lay$. It is easy to check that the following representation of $u_f^\lay$ is valid:
\begin{align*}
    u_f^\lay(t,x',y) = \sum_{i=1}^3 \int_0^t f^\lay_i(s,x')\partial_t\theta_i(t-s,x',y)\s 
\end{align*}
for almost every $(t,x',y) \in S\times \omega \times Y_0$, where $\theta$ is the unique weak solution to the cell problem
\begin{align}
\begin{aligned}\label{eq:cell_problem_theta}
    \partial_t(\varrho^\lay\partial_t\theta_i) - \nabla_y\cdot(A^\lay e_y(\theta_i)) & =0 &\mbox{ in }& S\times\omega\times Y_0,\\
    -A^\lay e_y(\theta)\nu &=0 &\mbox{ on }& S \times\omega\times\partial Y_1,\\
    \theta_i &=0 &\mbox{ on }& S\times\omega\times(S^+_Y \cup S^-_Y),\\
    \theta_i(0) &=0 &\mbox{ in }& \omega\times Y,\\
    \partial_t\theta_i(0) &=e_i &\mbox{ in }& \omega\times Y.
\end{aligned}
\end{align}
In summary, we have the following representation for $u^\lay$:
\begin{proposition}\label{prop:decomposition_u^M}
Under the additional assumptions \ref{ass:additional_solution_trace} and \ref{ass:additional_initial_value}, the limit function $u^\lay$ from Proposition \ref{prop:compactness_layer_gamma_1} fulfils 
\begin{align*}
u^\lay(t,x',y) =& u_0^\lay + \tilde{u}_1^\lay + \sum_{i=1}^3 \int_0^t f^\lay_i(s,x')\partial_t\theta_i(t-s,x',y)\s 
\\
&+ \sum_{\pm}\sum_{i=1}^3 \bigg\{\int_0^t \partial_t u_i^{\pm} (s,0,x') \chi_i^{\pm}(t-s,x',y) \s + \int_0^t \partial_{tt} u_i^{\pm}(s,0,x') \eta_i^{\pm}(t-s,x',y)\s\\
&+ \partial_t u_i^{\pm}(0,0,x') \eta_i^{\pm}(t,x',y)\bigg\}
\end{align*}
for almost every $(t,x',y) \in S\times \omega \times Y_0$,
where the cell solutions $\tilde{u}_1^\lay$, $\theta_i$, $\chi_i^{\pm}$ and $\eta_i^{\pm}$ for $i=1,2,3$ are given by the problems  $\eqref{eq:cell_problem_tilde_u1^M}$, $\eqref{eq:cell_problem_chi^pm}$, $\eqref{eq:cell_problem_eta^pm}$  and $\eqref{eq:cell_problem_theta}$, respectively.
\end{proposition}

We note that all cell problems necessary for the representation of $u^\lay$ can be calculated independently of the bulk solutions $u^{\pm}$. 
As a final step, we eliminate $u^\lay$ in the equation $\eqref{eq:macro_model_two_scale_gamma_1}$ in order to derive the effective coefficients for the interface condition across $\omega$; more precisely, we arrive at a condition for the normal stresses on $\omega$ with respect to $\Omega^{\pm}$. In order to make the following calculations less technical, we make an additional assumption:
\begin{enumerate}
[label = (A\arabic*$^{\prime}$)]
\setcounter{enumi}{2}
\item\label{ass:additional_initial_zero} The initial conditions and the forcing in the layer vanish, i.e.~$u_0^\lay = u_1^\lay = 0$ and $f^\lay = 0$; in particular, this implies $u^{\pm}_0|_{\omega}= 0$.
\end{enumerate}
Nevertheless, the following procedure can also be carried out for inhomogeneous data. 

Using \ref{ass:additional_initial_zero} in the representation of Proposition \ref{prop:decomposition_u^M}, we obtain $u^\lay = w^\lay$, i.e.
\begin{align*}
u^\lay(t,x',y) = \sum_{\alpha \in \{\pm\}} \sum_{i=1}^3 \left\{ \int_0^t \partial_t u_i^{\alpha} (s,0,x') \chi_i^{\pm}(t-s,x',y) \s + \int_0^t \partial_{tt} u_i^{\pm}(s,0,x') \eta_i^{\alpha} (t-s,x',y) \s \right\}.
\end{align*}
For the remainder of this subsection, we use the notation $\sum_{\alpha \in\{\pm\}}$ instead of $\sum_{\pm}$, since we have to deal with double sums over $\{\pm\}$ in what follows. We have
\begin{align*}
 \partial_t u^\lay(t,x',y) =& \sum_{\alpha \in \{\pm\}} \sum_{i=1}^3 \bigg\{   \partial_t u_i^{\alpha}(s,0,x')\phi_i^{\alpha}(y) + \int_0^t \partial_t u_i^{\alpha}(s,0,x') \partial_t \chi_i^{\alpha} (t-s,x',y) \s \\
   &+ \int_0^t \partial_{tt} u_i^{\alpha}(s,0,x') \partial_t \eta_i^{\alpha}(t-s,x',y) \s\bigg\}.
\end{align*}
To obtain an effective equation for $u^+$ and $u^-$ without coupling to the limit displacement $u^\lay$ in the interface, we plug the expression for $u^\lay$ into the two-scale homogenised equation $\eqref{eq:macro_model_two_scale_gamma_1}$ and choose suitable test functions $(v^+,v^\lay,v^-)$; more precisely, for given $v^{\pm} \in L^2(S;H^1_{\GamD\cap \partial \Omega^{\pm}}(\Omega^{\pm}))^3$ we define 
\begin{align*}
    v^\lay(t,x',y) = \sum_{\beta\in \{\pm\}} \sum_{j=1}^3 \chi_j^{\beta}(t,x',y) v_j^{\beta}(t,0,x').
\end{align*}
Obviously, we have $(v^+,v^\lay,v^-) \in L^2(S,\spaceH_1)$. This choice of test function carries all the necessary information about the effective model for $u^+$ and $u^-$. 
For the sake of simplicity (again to make the calculations less technical), we assume $\partial_t(\varrho^\lay u^\lay) \in L^2(S\times \omega \times Y_0)^3$ and, since we are mainly interested in the interface condition on $\omega$, we assume the bulk condition as smooth as necessary. Hence, choosing $v^{\pm}(0) = v^{\pm}(T) = 0$ (this is inherited by $v^\lay$), we can write $\eqref{eq:macro_model_two_scale_gamma_1}$ in the following way:
\begin{align}
\label{eq:aux_effecive_equation}
\sum_{\beta \in \{\pm\}} \int_{\omega} A^{\beta} e(u^{\beta}) \nu^{\beta} \cdot v^{\beta} d\sigma + \int_{\omega \times Y_0} \partial_t (\varrho^\lay \partial_t u^\lay) \cdot v^\lay + A^\lay e_y(u^\lay) : e_y(v^\lay)\x' \y = 0. 
\end{align}
This equation is valid in the distributional sense almost everywhere in $S$. For the second time derivative of $u^\lay$, we have with the representation of $u^\lay$ (and using the properties of $\chi_i^{\pm}$ and $\eta_i^{\pm}$):
\begin{align*}
\partial_t(\varrho^\lay \partial_t u^\lay)= \sum_{\alpha \in \{\pm\}} \sum_{i=1}^3 \bigg\{ \int_0^t& \varrho^\lay \partial_t u_i^{\alpha} (s,0,x') \partial_{tt} \chi_i^{\alpha} (t-s,x',y) \s
\\
&+ \int_0^t  \partial_{tt} u_i^{\alpha}(s,0,x') \partial_t (\varrho^\lay \partial_t \eta_i^{\alpha})(t-s,x',y) \s \bigg\}.
\end{align*}
Let us now define the effective coefficients  for $i,j=1,2,3$ and $\alpha,\beta \in \{\pm\}$:
\begin{align}
\label{def:effective_G}
G_{ji}^{\alpha\beta} (\tau,t,x') &:= \int_{Y_0} \partial_t(\varrho^\lay \partial_t \chi_i^{\alpha})(\tau,x',y) \cdot \chi_i^{\beta}(t,x',y) + A^\lay e_y(\chi_i^{\alpha})(\tau,x',y): e_y(\chi_j^{\beta})(t,x',y) \y,
\\
\label{def:effecitve_F}
F_{ji}^{\alpha\beta}(\tau,t,x') &:= \int_{Y_0} \partial_t(\varrho^\lay \partial_t \eta_i^{\alpha})(\tau,x',y) \cdot \chi_i^{\beta}(t,x',y) + A^\lay e_y(\eta_i^{\alpha})(\tau,x',y): e_y(\chi_j^{\beta})(t,x',y) \y.
\end{align}
Using the representations of $u^\lay$ and $v^\lay$ in $\eqref{eq:aux_effecive_equation}$ and using the previous calculations and definitions of the coefficients $G^{\alpha\beta}$ and $F^{\alpha\beta}$, we obtain (for almost every $t \in (0,T)$)
\begin{align*}
\sum_{\beta \in \{\pm\}} &\int_{\omega} A^{\beta} e(u^{\beta})(t,0,x') \nu^{\beta} \cdot v^{\beta}(t,0,x') \mathrm{d}\sigma 
\\
&+  \sum_{\alpha,\beta\in \{\pm\}} \int_{\omega} \int_0^t G^{\alpha\beta}(t-s,t,x')\partial_t u^{\alpha}(s,0,x') + F^{\alpha\beta}(t-s,t,x') \partial_{tt}u^{\alpha}(s,0,x') \s \cdot v^{\beta} \x' = 0.
\end{align*}
In other words, the normal stress of the bulk displacements $u^{\pm}$ on $\omega$ is given by
\begin{align}
\begin{aligned}
   \left( A^{\beta}e(u^{\beta}) \nu^{\beta} \right)(t,x') &= \sum_{\alpha \in\{\pm\}}  \int_0^t G^{\alpha\beta}(t-s,t,x')\partial_t u^{\alpha}(s,0,x') + F^{\alpha\beta}(t-s,t,x') \partial_{tt}u^{\alpha}(s,0,x') \s
   \\
   &=: H^{\beta}(t,x',u^+,u^-)
   .
\end{aligned}
\end{align}

In summary, we have shown the following result:
\begin{theorem}
The bulk displacements $u^+$ and $u^-$ are a weak solution to the effective problem $\eqref{eq:macro_model_gamma1_effective}$.
\end{theorem}

\subsection{The case $\gamma = -1$}

We recall from Proposition \ref{prop:compactness_layer_gamma_-1} that the limit function $u^\lay$ in the interface is independent of the microscopic variable in the case $\gamma = -1$ and we expect an additional macroscopic equation on the interface $\omega$ describing the movement of the interface (the displacements $u^+$ and $u^-$ are continuous across $\omega$, see Proposition \ref{prop:interface_condition}). Due to the heterogeneous structure of the interface, the macroscopic equation on $\omega$ includes a homogenised elasticity coefficient depending on suitable cell problems which are introduced now: 
Let $\chi_{ij}^A \in \widetilde{H}_{\#}^1(Y_0)^3/\R^3$ for $i,j=1,2,3$ be the unique weak solution to 
\begin{align}
\begin{aligned}\label{eq:CellProblem_chi_ij}
-\nabla_y \cdot (A (e_y(\chi_{ij}^\mathrm{A}) + M_{ij})) &= 0 &\mbox{ in }& Y_0,
\\
-A(e_y(\chi_{ij}^\mathrm{A}) + M_{ij} ) \nu &= 0 &\mbox{ on }& \partial Y_1 \cup S_Y^+ \cup S_Y^-,
\\
\chi_{ij}^\mathrm{A} \mbox{ is } Y'\mbox{-periodic, } & \frac{1}{\abs{Y_0}}\int_{Y_0} \chi_{ij}^\mathrm{A} \y = 0,
\end{aligned}
\end{align}
with $M_{ij} \in \R^{3\times 3}$ defined by 
\begin{align*}
    M_{ij}:= \frac12 (e_i \otimes e_j + e_j \otimes e_i).
\end{align*}
The Korn inequality for $Y'$-periodic functions with vanishing mean implies the existence of a unique weak solution to this problem. Now, we show a representation of the limit function $u^{\lay,1}$ (the first-order corrector):
\begin{proposition}\label{prop:representation_uM1_gamma-1}
Let $u^\lay$ and $u^{\lay,1}$ be the limit functions from Proposition \ref{prop:compactness_layer_gamma_-1}. Then, the following representation holds for almost every $(t,x',y) \in S \times \omega \times Y_0$:
\begin{align*}
    u^{\lay,1}(t,x',y) = \sum_{i,j=1}^2 e_{x'}(u^\lay)_{ij}(t,x') \chi_{ij}^\mathrm{A}(y).
\end{align*}
\end{proposition}
\begin{proof}
We argue in a similar way as in the proof of \cite[Theorem 5.2]{bhattacharya2022homogenization}.  We define
\begin{align*}
 v^\eps(t,x)=\begin{cases}
        v^\lay\left(t,x',\frac{x'}{\eps},1\right) \varrho\left(\frac{x_3}{\eps}\right), &(t,x)\in S\times\OmEps^{+},\\
        v^\lay \left(t,x',\frac{x}{\eps}\right), &(t,x)\in S\times\OmEps^\lay,\\
         v^\lay\left(t,x',\frac{x'}{\eps},-1\right) \varrho\left(-\frac{x_3}{\eps}\right), &(t,x)\in S\times\OmEps^{-},
    \end{cases}
\end{align*}
with $v^\lay \in C_0^{\infty}(S \times \omega,C^{\infty}(\overline{Y_0}))^3$ $Y'$-periodic and $\varrho \in C_0^{\infty}([1,2))$ such that $0 \le \varrho \le 1$ and $\varrho = 1$ in $\left[1,\frac32\right]$. This is an admissible test function for the microscopic equation $\eqref{eq:weak_formulation}$ and we obtain using the unfolding operator (similar as in the proof of Theorem \ref{thm:main_convergence_two_scale_model})  writing $v^{\eps,\pm}:= v^{\eps}|_{\OmEps^{\pm}}$ 
\begin{align}
\begin{aligned}\label{eq:aux_representation}
-&\int_0^T\int_{\Omega^{+}} \varrho^{+}\chi_{\OmEps^{+}}\partial_t \ueps\cdot\partial_t v^{\eps,+}\,\x\t -\int_0^T\int_{\Omega^{-}}\varrho^{-}\chi_{\OmEps^{-}}\partial_t \ueps\cdot\partial_t v^{\eps,-}\,\x\t\\
    &-\int_0^T\int_{\omega\times Y_0} \TEps(\varrho^\lay_\eps)\TEps(\partial_t \ueps)\cdot\TEps(\partial_t v^\eps)\,\x'\y\t +\int_0^T\int_{\Omega^{+}} A^{+} \chi_{\OmEps^{+}}e(\ueps):e(v^{\eps,+})\,\x\t\\
    &+\int_0^T\int_{\Omega^{-}} A^{-} \chi_{\OmEps^{-}}e(\ueps):e(v^{\eps,-})\,\x\t  + \int_0^T\int_{\omega\times Y_0} \TEps(A^\lay_\eps)\TEps(e(\ueps)):\TEps(e(v^\eps))\,\x'\y\t\\
	=& \int_0^T\int_{\Omega^{+}} \varrho^{+}\chi_{\OmEps^{+}} f^{+}\cdot v^{\eps,+}\,\x\t +\int_0^T\int_{\Omega^{-}} \varrho^{-}\chi_{\OmEps^{-}} f^{-}\cdot v^{\eps,-}\,\x\t.
\end{aligned}
\end{align}
After multiplying this equation with $\eps$, we pass to the limit $\eps \to 0$. The only critical terms are the stress terms including the symmetrised gradients $e(\ueps)$. First, we show that the bulk term vanishes in the limit. We have 
\begin{align*}
\left| \eps \int_0^T \int_{\OmEps^{\pm}} A^{\pm} e(\ueps) : e(v^{\eps,\pm}) \x \t \right| \le   C \|e(\ueps)\|_{L^2(S\times \OmEps^{\pm})} \left[ \int_{\eps}^{2\eps} \left|\varrho\left(\frac{x_3}{\eps}\right)\right|^2 +  \left|\varrho'\left(\frac{x_3}{\eps}\right)\right|^2 \x_3 \right]^{\frac12} \le C \sqrt{\eps}
\end{align*}
using the a priori estimates from Theorem \ref{theo:existence_result_gen}.
Hence, for $\eps \to 0$ all terms except the stress term in the layer vanish and we obtain with the compactness result for $e(\ueps)$ from Proposition \ref{prop:compactness_layer_gamma_-1} that
\begin{align*}
    \int_0^T \int_{\omega} \int_{Y_0} A^\lay \left[ e_{x'}(\hat{u}^\lay) + e_y(u^{\lay,1})\right] : e_y(v^\lay) \y \x'\t = 0,
\end{align*}
having written $\hat{u}^\lay := (u_1^\lay,u_2^\lay)$. 
By density, this equation is valid for all $v^\lay \in L^2(S \times \omega;\widetilde{H}_{\#}^1(Y_0))^3$. Since the solution to this equation is unique (up to a constant) by the Lax--Milgram theorem and the Korn inequality, we get the desired result.
\end{proof}

The previous proof shows that $u^{\lay,1}$ is the unique (up to a constant) weak solution to the cell problem
\begin{align}
\begin{aligned}\label{cell_problem_uM_gamma_-1}
-\nabla_y \cdot \left( A^\lay\left[ e_{x'}(\hat{u}^\lay) + e_y (u^{\lay,1}) \right] \right) &= 0 &\mbox{ in }& S \times \omega \times Y_0,
\\
- A^\lay\left[ e_{x'}(\hat{u}^\lay) + e_y (u^{\lay,1}) \right] \nu &= 0 &\mbox{ on } & S\times \omega \times (\partial Y_1 \cup S_Y^+ \cup S_Y^-),
\\
u^{\lay,1} \mbox{ is } Y'\mbox{-periodic}.
\end{aligned}
\end{align}
Next, we give the proof of the main result for the case $\gamma = -1$. We recall  the definition of the solution space \eqref{def:spaceH_-1} (to which the limit function $(u^+,u^\lay,u^-)$ from Proposition \ref{prop:compactness_layer_gamma_-1} belongs),
\begin{align*}
\spaceH_{-1}= \left\{\phi \in H_{\Gamma_D}^1(\Omega)^3 \, : \, (\phi_1,\phi_2)|_{\omega} \in H_0^1(\omega)^2 \right\}.
\end{align*}
For a function $\phi\in \spaceH_{-1}$, we use the notation $(\phi^+,\phi^\lay,\phi^-)$ with $\phi^{\pm}:= \phi|_{\Omega^{\pm}} $ and $\phi^\lay := \phi|_{\omega}$. Before we give the proof of the derivation of the macro model, we show the following density result (see \cite[Lemma 5.3]{GahnEffectiveTransmissionContinuous} for a similar result with different boundary conditions):
\begin{lemma}\label{lem:density}
Functions $\phi \in C_0^{\infty}(\overline{\Omega}\setminus \Gamma_D)^3 $ with $(\phi_1,\phi_2)|_{\omega} \in C_0^{\infty}(\omega)^2$ are dense in $\spaceH_{-1}$ with respect to the norm induced by the inner product (as usual, we denote with $\hat{\cdot}$ the last two components of a vector field)
\begin{align*}
(u,v)_{\spaceH_{-1}} := (u,v)_{H^1(\Omega)} + (\hat{u},\hat{v})_{H^1(\omega)} + (u_1,v_1)_{L^2(\omega)}
\end{align*}
on $\spaceH_{-1}$.
\end{lemma}
\begin{proof}
We define the space
\begin{align*}
    \spaceH_{-1}^{\infty} := \left\{\phi \in C_0^{\infty}(\overline{\Omega}\setminus \Gamma_D)^3  \, : \, \hat{\phi}|_{\omega} \in C_0^{\infty}(\omega)^2 \right\}
\end{align*}
and have to show that this space is dense in $\spaceH_{-1}$. We denote the closure of $\spaceH_{-1}^{\infty}$ by $Y$. Hence, we have $\spaceH_{-1} = Y + Y^{\perp}$ and we have to show that $Y^{\perp} = \{0\}$. Let $u \in Y^{\perp}$. Then, for all $v \in \spaceH_{-1}^{\infty}$ it holds that
\begin{align*}
0 = (u,v)_{H^1(\Omega)} + (u|_{\omega},v|_{\omega})_{H^1(\omega)} + (u_1,v_1)_{L^2(\omega)}.
\end{align*} 
In particular, we can choose $v \in C_0^{\infty}(\Omega^{\pm})^3$ (extended by zero to the whole domain $\Omega$) to find
\begin{align*}
0 = (u^{\pm}, v^{\pm})_{H^1(\Omega^{\pm})}
\end{align*}
and, therefore, $\Delta u^{\pm} = u^{\pm}$ in $\Omega^{\pm}$ in the distributional sense. This also implies the existence of the normal trace $\nabla u^{\pm} \nu $ in $H^{-\frac12}(\partial \Omega^{\pm})^3$.  By integration by parts, we obtain for  every $v \in \spaceH_{-1}^{\infty}$ 
\begin{align}\label{eq:density_aux}
0 = (u,v)_{\spaceH_{-1}} = (\hat{u},\hat{v})_{H^1(\omega)} + (u_1,v_1)_{L^2(\omega)} + \sum_{\pm} \langle \nabla u^{\pm}\nu,v\rangle_{H^{-\frac12}(\partial \Omega^{\pm}),H^{\frac12}(\partial \Omega^{\pm})}.
\end{align}
Hence, we have reduced the problem to an equation on the boundary and we now have to show that we can choose $v = u$. For this, we first note that by the inverse trace theorem there exists $\tilde{u}^{\pm} \in H^1(\Omega^{\pm})^3$ such that 
\begin{align*}
    \tilde{u}^{\pm}|_{\partial \Omega^{\pm}} = \begin{cases}
        u^{\pm}|_{\partial \Omega^{\pm}} &\mbox{ on } \partial \Omega^{\pm} \setminus \omega,
        \\
        0 &\mbox{ on } \omega.
    \end{cases}
\end{align*}
Here, we used the fact that the right-hand side in the equation above is an element in $H^{\frac12}(\partial \Omega^{\pm})^3$ as $u \in \spaceH_{-1}$ (extend $u|_{\omega}$ constantly to the bulk domains and multiply it by a cut-off function, and then subtract it from $u$). Hence, there exists a sequence $\phi_n \in C_0^{\infty}(\overline{\Omega}\setminus \Gamma_D)^3$ such that $\phi_n^{\pm} \rightarrow \tilde{u}^{\pm}$ in $H^1(\Omega^{\pm})^3$ and also the traces converge in $H^{\frac12}(\partial \Omega^{\pm})^3$ to the traces of $\tilde{u}^{\pm}$. Now, since $\hat{u} \in H_0^1(\omega)^2$, there exists a sequence $\psi_n \in C_0^{\infty}(\omega)^3$ with $\hat{\psi}_n \rightarrow \hat{u} $  in $H^1(\omega)^2$ and $[\psi_n]_1 \rightarrow u_1$ in $L^2(\omega)$. We extend the sequence $\psi_n$ to a sequence in $C_0^{\infty}(\Omega)^3$ (constant in $x_1$-direction and multiplied by a cut-off function) and use the same notation for the extended sequence. Now, we define the sequence
\begin{align*}
    v_n := \phi_n + \psi_n \in \spaceH_{-1}
\end{align*}
and use it as a test function in equation  $\eqref{eq:density_aux}$. Since $v_n|_{\partial \Omega^{\pm}} \rightarrow u^{\pm}|_{\partial \Omega^{\pm}}$ in $H^{\frac12}(\partial \Omega^{\pm})^3$ and $v_n|_{\omega} \rightarrow u|_{\omega} $ in $H^1(\omega)^2 \times L^2(\omega))$, we get 
after integration by parts in the boundary terms and using $\Delta u^{\pm} = u^{\pm}$
\begin{align*}
    0 = \|\hat{u}\|_{H^1(\omega)}^2 + \|u_1\|_{L^2(\omega)}^2 + \|u\|^2_{H^1(\Omega)},
\end{align*}
and therefore $u = 0$, which implies the desired result.
\end{proof}
Now, we can give the proof of the main result.

\begin{theorem}
The limit function $(u^+,u^\lay,u^-)$ from Proposition \ref{prop:compactness_bulk_domains} and \ref{prop:compactness_layer_gamma_-1} is the unique weak solution to the macroscopic problem $\eqref{Macro_Model_gamma_-1}$.
\end{theorem}
\begin{proof}
As a test function in $\eqref{eq:weak_formulation}$, we choose $v \in L^2(S;\spaceH_{-1}^{\infty})$ with $\partial_t v \in L^2(S;L^2(\Omega))^3$ and $v(T) = 0$ (see the proof of Lemma \ref{lem:density} for the definition). After unfolding and using the convergence results from Section \ref{sec:compactness}, we obtain for $\eps \to 0$
\begin{align}
\begin{aligned}\label{eq:macro_model_two_scale}
	-&\int_0^T\int_{\Omega^{+}} \varrho^{+}\partial_t u^{+}\cdot\partial_t v^{+}\,\x\t -\int_0^T\int_{\Omega^{-}}\varrho^{-}\partial_t u^{-}\cdot\partial_t v^{-}\,\x\t\\
    &-\int_0^T\int_{\omega\times Y_0} \varrho^\lay \partial_t u^\lay\cdot\partial_t v^\lay\,\x'\y\t +\int_0^T\int_{\Omega^{+}} A^{+}e(u^{+}):e(v^{+})\,\x\t\\
    &+\int_0^T\int_{\Omega^{-}} A^{-}e(u^{-}):e(v^{-})\,\x\t +\int_0^T\int_{\omega\times Y_0} A^\lay\left[e_{x'}(\hat{u}^\lay) + e_y(u^{\lay,1})\right] : e_x(v)(0,x')\,\x'\y\t\\
	=& \int_0^T\int_{\Omega^{+}} \varrho^{+}f^{+}\cdot v^{+}\,\x\t +\int_0^T\int_{\Omega^{-}} \varrho^{-} f^{-}\cdot v^{-}\,\x\t +\int_0^T\int_{\omega\times Y_0} \varrho^\lay(y) f^\lay\cdot v^\lay\,\x'\y\t\\
    &+\int_0^T\int_{\GamN\cap\partial\Omega^{+}} g\cdot v^{+}\,\mathrm{d}\sigma(x)\t +\int_0^T\int_{\GamN\cap\partial\Omega^{-}} g\cdot v^{-}\,\mathrm{d}\sigma(x)\t +\int_{\Omega^{+}}\varrho^{+}u^{+}_1\cdot v^{+}(0,x)\,\x\\
    &+\int_{\Omega^{-}}\varrho^{-}u^{-}_1\cdot v^{-}(0,x)\,\x +\int_{\omega\times Y_0} \varrho^\lay u^\lay_1(x',y)\cdot v(0,x')\,\x'\y.
\end{aligned}
\end{align}
By density, see Lemma \ref{lem:density}, this equation is valid for all $v $ having values in the space $\spaceH_{-1}$.
Let us consider in more detail the stress term over $\omega \times Y_0$. Using the decomposition  $e(v) = e_{x'}(\hat{v}) + 2 \sum_{i=1}^3 e_x(v)_{i3} M_{i3}$ with
\begin{align*}
  e_{x'}(\hat{v}):=    \begin{pmatrix}
      0 & 0& 0 \\
      0 & \partial_2 v_2 & \frac12 (\partial_2 v_3 + \partial_3 v_2) \\
      0 & \frac12 (\partial_2 v_3 + \partial_3 v_2) & \partial_3 v_3
    \end{pmatrix},
\end{align*}
we obtain (almost everywhere in $(0,T)$)
\begin{align*}
\int_{\omega\times Y_0}& A^\lay\left[e_{x'}(\hat{u}^\lay) + e_y(u^{\lay,1})\right] : e_x(v)(0,x')\,\x'\y  
\\
=&\int_{\omega} \int_{Y_0} A^\lay\left[e_{x'}(\hat{u}^\lay) + e_y(u^{\lay,1})\right] : e_{x'}(v) (0,x') \y\x' 
\\
&+2 \sum_{i=1}^3 \int_{\omega}e_x(v)_{i3} \int_{Y_0} A^\lay\left[e_{x'}(\hat{u}^\lay) + e_y(u^{\lay,1})\right] : M_{i3} \y \x'. 
\end{align*}
The second term is equal to zero, since we have that $y \mapsto y_3 e_i$ is a $Y'$-periodic function with $e_y(y_3 e_i) = M_{i3}$, which can be used as a test function in the cell problem $\eqref{cell_problem_uM_gamma_-1}$. Hence, after an elementary calculation we get
\begin{align*}
\int_{\omega\times Y_0}& A^\lay\left[e_{x'}(\hat{u}^\lay) + e_y(u^{\lay,1})\right] : e_x(v)(0,x')\,\x'\y   = \int_{\omega} A^{\ast} e_{x'} (\hat{u}) : e_{x'}(\hat{v}^\lay) \x',
\end{align*}
with the effective elasticity tensor $A^{\ast} \in \R^{2\times 2 \times 2 \times 2}$ defined by 
\begin{align}
\begin{aligned}\label{def:effective_elasticity_tensor}
A^{\ast}_{ijkl}:&=\int_{Y_0} [A(M_{kl} + D_y(\chi_{kl}^\mathrm{A}))]:[M_{ij} + D_y(\chi_{ij}^\mathrm{A})]\y , \quad i,j,l,k = 1,2,
\end{aligned}
\end{align}
with the cell solutions $\chi_{kl}^\mathrm{A}$ from $\eqref{eq:CellProblem_chi_ij}$.
This implies that $(u^+,u^\lay,u^-)$ is a weak solution to the macroscopic problem $\eqref{Macro_Model_gamma_-1}$. Uniqueness follows by standard arguments and we skip the details.
\end{proof}

\subsection{The case $\gamma = -3$}

Finally, we consider the case $\gamma = -3$. Similar results were obtained in \cite{gahn2022derivation} and \cite{gahn2025b_effective} for fluid--structure interaction problems. Here, we adapt these methods to the purely elastic case. We use the compactness results from Proposition \ref{prop:compactness_layer_gamma_-3} and first give a representation for $u^{\lay,2}$. For this purpose, we introduce yet another cell problem: Find $\chi_{ij}^\mathrm{B} \in \widetilde{H}_{\#}^1(Y_0)^3$, $i=1,2,3$, which is the weak solution to 
\begin{align}
\begin{aligned} \label{cell_problem_chi_ij_B}
-\nabla_y \cdot (A^\lay (e_y(\chi_{ij}^\mathrm{B}) - y_1 M_{ij})) &= 0 &\mbox{ in }& Y_0,
\\
-A^\lay (e_y(\chi_{ij}^\mathrm{B}) - y_1 M_{ij})\nu &= 0 &\mbox{ on }& \partial Y_1 \cup S_Y^+ \cup S_Y^-,
\\
\chi_{ij}^\mathrm{B} \mbox{ is } Y\mbox{-periodic, } & \frac{1}{\abs{Y_0}}\int_{Y_0} \chi_{ij}^\mathrm{B} \y = 0.
\end{aligned}   
\end{align}
\begin{proposition}
Let $u^\lay$, $\hat{u}^\lay$ and $u^{\lay,2}$ be the limit functions from Proposition \ref{prop:compactness_layer_gamma_-3}. Then, there holds for almost every $(t,x',y) \in S\times \omega \times Y_0$
\begin{align*}
    u^{\lay,2}(t,x',y) = \sum_{i,j=1}^2 \left[ e_{x'}(\hat{u}^\lay)_{ij}(t,x') \chi_{ij}^\mathrm{A}(y) + \partial_{ij} [u^3]_1(t,x') \chi_{ij}^\mathrm{B}(y) \right],
\end{align*}
where the cell solutions $\chi_{\ij}^\mathrm{A}$ and $\chi_{ij}^\mathrm{B}$ are given via $\eqref{eq:CellProblem_chi_ij}$ and $\eqref{cell_problem_chi_ij_B}$, respectively.
\end{proposition}
\begin{proof}
We use similar arguments as in the proof of Proposition \ref{prop:representation_uM1_gamma-1}, so we skip some details. As a test function in $\eqref{eq:weak_formulation}$, we use
\begin{align*}
 v^\eps(t,x)=\begin{cases}
       \eps^2  v^\lay\left(t,x',\frac{x'}{\eps},1\right) \varrho\left(\frac{x_3}{\eps}\right), &(t,x)\in S\times\OmEps^{+},\\
      \eps^2  v^\lay \left(t,x',\frac{x}{\eps}\right), &(t,x)\in S\times\OmEps^\lay,\\
        \eps^2 v^\lay\left(t,x',\frac{x'}{\eps},-1\right) \varrho\left(-\frac{x_3}{\eps}\right), &(t,x)\in S\times\OmEps^{-},
    \end{cases}
\end{align*}
with $v^\lay \in C_0^{\infty}(S \times \omega,C^{\infty}(\overline{Y_0}))^3$ $Y'$-periodic and $\varrho \in C_0^{\infty}([1,2))$ such that $0 \le \varrho \le 1$ and $\varrho = 1$ in $\left[1,\frac32\right]$. Now, the only difference to equation $\eqref{eq:aux_representation}$ is the stress term which now takes the form
\begin{align*}
\frac{1}{\eps^2} \int_0^T \int_{\omega \times Y_0} \TEps (A_{\eps}^\lay) \TEps(e(\ueps)) : \TEps(e(v^{\eps})) \y \x'  \t.
\end{align*}
Again, this is the only term not vanishing for $\eps \to 0$ and we obtain with Proposition \ref{prop:compactness_layer_gamma_-3}
\begin{align*}
\int_0^T \int_{\omega} \int_{Y_0} A^\lay \left[ e_{x'}(\hat{u}^\lay) - y_3 \nabla^2_{x'} [u^\lay]_1 + e_y(u^{\lay,2})  \right] : e_y (v^\lay) \y \x' \t = 0,
\end{align*}
which holds for all $v^\lay  \in L^2(S \times \omega;\widetilde{H}_{\#}^1(Y_0))^3 $ by density. This implies the desired result.
\end{proof}
Recalling the solution space for the bulk displacements,
\begin{align*}
\spaceH_{-3}= \left\{ \phi \in H^1(\Omega)^3 \, : \, \phi = 0 \mbox{ on } \Gamma_D, \, \, \phi|_{\omega} = (\phi_3,0,0)|_{\omega} \in H_0^2(\omega)^3 \right\},
\end{align*}
see \eqref{def:spaceH_-3}, we immediately obtain $(u^+,u^-) \in \spaceH_{-3}$ with $u^{\pm} = u^\lay$ on $\omega$ from Proposition \ref{prop:interface_condition}.
Now, we are able to prove the main result for the case $\gamma = -3$.
\begin{theorem}
The limit function $(u^+,u^\lay,\hat{u}^\lay,u^-)$ from Proposition \ref{prop:compactness_bulk_domains} and \ref{prop:compactness_layer_gamma_-3} is the unique weak solution to the macroscopic model $\eqref{Macro_Model_gamma_-3}$.
\end{theorem}
\begin{proof}
Let $\psi \in C_0^{\infty}([0,1))$ be a cut-off function with $0 \le \psi \le 1$ and $\psi (0) = 1$. Further, let $V \in L^2(S;\spaceH_{-3})$ such that $\partial_t V \in L^2(S \times \Omega)^3$ and $\bar{U} := (U_2,U_3) \in L^2(S;H_0^1(\omega))^2$  with $\partial_t \bar{U} \in  L^2(S;L^2(\omega))^2$ and we do not distinguish the notation for the canonical embedding into $\{0\} \times H_0^1(\omega)^2$, i.e.~$\bar{U} = (0,U_2,U_3)$. Further, we assume $V(T) = 0$ and $\bar{U}(T) = 0$.
As a test function in $\eqref{eq:weak_formulation}$, we choose
\begin{align*}
v^{\eps}(t,x):= \begin{cases}
    V(t,x)  + \eps \psi\left( \frac{x_1 \mp \eps}{\eps}\right) \left( \bar{U}(t,x') - \nabla_{x'} V_3(t,x') \right) &\mbox{ for } (t,x) \in S \times \OmEps^{\pm},
    \\
    V_1(t,x') e_1 + \eps  \left( \bar{U}(t,x') - \frac{x_1}{\eps}\nabla_{x'} V_3(t,x') \right) &\mbox{ for } (t,x) \in S \times \OmEps^\lay.
\end{cases}
\end{align*}
Here, $V_3(t,x')$ denotes the trace of $V$ on $\omega$. The crucial point is that we have $e(v^{\eps}) = \eps( e_{x'}(\hat{U}) - \nabla^2_{x'} V_1)$ in $\OmEps^\lay$. Hence, after a long but straightforward calculation (we refer to \cite[Section 6]{gahn2022derivation} for analogous details), we get for $\eps \to 0$ using Proposition \ref{prop:compactness_layer_gamma_-3}
\begin{align*}
- \sum_{\pm}& \left\{ \int_0^T \int_{\Omega^{\pm}} \varrho^{\pm} \partial_t u^{\pm} \cdot \partial_t V \x \t  - \int_0^T \int_{\Omega^{\pm}} A^{\pm} e(u^{\pm}) : e(V) \x \t \right\}
\\
-& \int_0^T \int_{\omega} \int_{Y_0} \varrho^\lay \partial_t [u^\lay]_1 \partial_t V_1 \y\x'\t \\
+& \int_0^T \int_{\omega} \int_{Y_0} A^\lay \left[ e_{x'}(\hat{u}^\lay)  - y_1 \nabla_{x'}[u^\lay]_1 + e_y(u^{\lay,2}) \right] : \left[ e_{x'}(\bar{U}) - y_1 \nabla_{x'}^2 V_1 \right] \y\x'\t
\\
&= \sum_{\pm} \left\{ \int_0^T \int_{\Omega^{\pm}} \varrho^{\pm} f^{\pm}   \cdot V \x \t  + \int_{\Omega^{\pm}} \varrho^{\pm} u_1^{\pm} \cdot V(0) \x + \int_0^T \int_{\Gamma_N} g \cdot V \mathrm{d}\sigma(x) \t\right\}
\\
&\hspace{3em} + \int_0^T \int_{\omega} \int_{Y_0} \varrho^\lay f_1^\lay V_1 \y \x' \t + \int_{\omega} \int_{Y_0} \varrho^\lay [u^\lay_1]_1 V_1 \y \x'.
\end{align*}
After another elementary calculation, we are left with 
\begin{align*}
\int_{Y_0} A^\lay &\left[ e_{x'}(\hat{u}^\lay)  - y_1 \nabla_{x'}[u^\lay]_1 + e_y(u^{\lay,2}) \right] : \left[ e_{x'}(\bar{U}) - y_1 \nabla_{x'}^2 V_1 \right] \y
\\
&= a^{\ast} e_{x'}(\hat{u}^\lay)  : e_{x'}(\bar{U}) + b^{\ast} \nabla_{x'}^2 [u^\lay]_1 : e_{x'}(\bar{U}) + b^{\ast} e_{x'} (\hat{u}^\lay) : \nabla_{x'}^2 V_1 + c^{\ast} \nabla_{x'}^2 [u^\lay]_1 : \nabla_{x'}^2 V_1
\end{align*}
with the effective coefficients $a^{\ast},b^{\ast},c^{\ast} \in \R^{2\times 2 \times 2 \times 2}$ defined by 
\begin{align}
\begin{aligned}
\label{def:effective_coefficients_plate}
a^{\ast}_{\alpha\beta\gamma\delta} &:=  \frac{1}{\vert Z^s \vert} \int_{Z^s} A  \left(D_y(\chi_{\alpha \beta}^\mathrm{A}) + M_{\alpha\beta}\right): \left(D_y (\chi_{\gamma\delta}^\mathrm{A})  + M_{\gamma\delta} \right)\y,
\\
b^{\ast}_{\alpha\beta\gamma\delta} &:=   \frac{1}{\vert Z^s \vert} \int_{Z^s} A \left(D_y(\chi_{\alpha \beta}^\mathrm{B})  - y_3 M_{\alpha\beta} \right) : \left(D_y (\chi_{\gamma\delta}^\mathrm{A})  + M_{\gamma\delta} \right)\y,
\\
c^{\ast}_{\alpha\beta\gamma\delta} &:=   \frac{1}{\vert Z^s \vert} \int_{Z^s} A  \left(D_y(\chi_{\alpha \beta}^\mathrm{B})  - y_3 M_{\alpha\beta} \right): \left(D_y (\chi_{\gamma\delta}^\mathrm{B})  -y_3 M_{\gamma\delta}\right)\y,
\end{aligned}
\end{align}
$\alpha,\beta,\gamma,\delta = 1,2$, 
where $\chi_{ij}^\mathrm{A}$ and $\chi_{ij}^\mathrm{B}$ are solutions to $\eqref{eq:CellProblem_chi_ij}$ and $\eqref{cell_problem_chi_ij_B}$, respectively. This implies the desired result.
\end{proof}

\begin{appendix}
\section{Weak solutions for wave equation in the reference cell}
\label{sec:appendix_wave_equation}

In this appendix, we briefly summarise some properties of weak solutions to the wave equation in the reference cell with mixed boundary conditions. Such results seem to be standard but since they are frequently used when giving the representation for $u^\lay$ in the case $\gamma  =1$, we summarise them here.

We consider the following  abstract cell problem: Let  $g \in L^2(S \times \omega \times Y_0)^3$, $h \in H^1(S; L^2(\omega \times \partial Y_0))^3$, $\mu_0 \in L^2(\omega; \widetilde{H}_{\#}^1(Y_0))^3$ and $\mu_1 \in L^2(\omega \times Y_0)^3$  be given such that $\mu_0|_{S_Y^{\pm}} = h(0)|_{S_Y^{\pm}} $. Then, there exists a unique weak solution $\mu$ to the problem 
\begin{align*}
 \partial_t (\varrho^\lay \partial_t \mu) - \nabla_y \cdot (A^\lay e_y(\mu)) &= g &\mbox{ in }& S \times \omega \times Y_0,
 \\
\mu &= h &\mbox{ on }& S \times \omega \times (S_Y^+ \cup S_Y^-),
\\
-A^\lay e_y(\mu) \nu &= 0 &\mbox{ on }& S\times \omega \times \partial Y_1,
\\
\mu(0) &= \mu_0 &\mbox{ in }& \omega \times Y_0,
\\
\partial_t \mu(1) &= \mu_1 &\mbox{ in }& \omega \times Y_0,
\\
\mu \mbox{ is } Y'\mbox{-periodic}.
\end{align*}
Here, we say that $\mu $ is a weak solution if
\begin{align*}
\mu \in H^1(S;L^2(\omega \times Y_0))^3 \cap L^2(S \times \omega; \widetilde{H}_{\#}^1(Y_0))^3
\end{align*}
with $\mu(0) = \mu_0$ and $\mu =h$ on $ S \times \omega \times (S_Y^+ \cup S_Y^-)$,
and for arbitrary $\tau \in [0,T]$ and all $v \in L^2((0,\tau)\times \omega \times Y_0)^3$ with $\partial_t v \in L^2((0,\tau) \times \omega \times Y_0)^3$ such that $v|_{S_Y^{\pm}}=0$ and $v(\tau) = 0$ the following equality is satisfied:
\begin{align*}
- \int_0^\tau &\int_{\omega} \int_{Y_0} \varrho^\lay\partial_t \mu \cdot \partial_t v \y \x' \s + \int_0^\tau \int_{\omega} \int_{Y_0} A^\lay e_y(\mu) : e_y(v) \y \x' \s
\\
&= \int_0^{\tau}\int_{\omega} \int_{Y_0} g \cdot v \y \x' \t +   \int_{\omega} \int_{Y_0}   \varrho^\lay \mu_1\cdot v(0,x',y) \y \x'.
\end{align*}
The existence of a unique weak solution, even with improved regularity for $\partial_t (\varrho^\lay \partial_t \mu)$, is standard. To compute a suitable representation of the limit function $u^\lay$, obtained through the homogenisation process in the thin layer, we use the following lemma. Although it appears to be standard, we could not find a specific reference in the literature; therefore, we present it here (without proof) for the sake of completeness.
\begin{lemma}\label{lem:aux_weak_sol_wave_equation}
Let  $f\in L^2((0,T)\times \omega)^3$. Then, for arbitrary $\tau\in [0,T]$ and all $v \in L^2((0,\tau)\times \omega \times Y_0)^3$ with $\partial_t v \in L^2((0,\tau)\times \omega \times Y_0)^3$ and $v(\tau)=0$, the following equality holds:
\begin{align*}
-\int_0^{\tau} & \int_{\omega} \int_{Y_0} \int_0^t \varrho^\lay f(s,x') \partial_t \mu(t-s,x',y) \cdot \partial_t v(t,x',y) \s\y \x'\t
\\
&+  \int_0^{\tau} \int_{\omega} \int_{Y_0} \int_0^t f( s,x')  A^\lay e_y(\mu)(t-s,x',y) : e_y(v)(t,x',y) \s\y \x'\t
\\
=& \int_0^{\tau} \int_{\omega} \int_{Y_0} \varrho^\lay f(t,x') \mu_1(x',y) \cdot v(t,x',y)  \y \x'\t
\end{align*}
\end{lemma}

\end{appendix}



\bibliographystyle{abbrv}
\bibliography{literature}

\end{document}